\documentclass[10pt]{article}
\usepackage{amsmath}
\usepackage{amssymb}
\usepackage{graphicx}
\usepackage{color}
\usepackage{amsthm}
\usepackage{amsfonts}
\usepackage{amscd}
\usepackage{mathrsfs}
\usepackage{bm}
\usepackage{enumerate}
\usepackage{algorithmic, algorithm}
\usepackage{leqno}

\usepackage[utf8x]{inputenc}
\usepackage{graphicx}
\usepackage{float}
\usepackage{latexsym}
\usepackage{color}

\newcommand{\B}{\mathbf}
\newcommand{\bx}{\mathbf{x}}
\newcommand{\bn}{\mathbf{n}}
\newcommand{\by}{\mathbf{y}}

\newcommand{\bs}{\mathbf{s}}
\newcommand{\p}{\partial}

\newcommand{\hk}{R_t(\bx, \by)}
\newcommand{\rhk}{\bar{R}_t(\bx, \by)}
\newcommand{\rrhk}{\bar{\bar{R}}_t(\bx, \by)}
\newcommand{\hkpipj}{R_t({\bf p}_i, {\bf p}_j)}

\newcommand{\rhkpipj}{\bar{R}_t({\bf p}_i, {\bf p}_j)}
\newcommand{\rhkpisj}{\bar{R}_t({\bf p}_i, {\bf s}_j)}
\newcommand{\hkxpj}{R_t(\bx, {\bf p}_j)}

\newcommand{\rhkxpj}{\bar{R}_t(\bx, {\bf p}_j)}
\newcommand{\rhkxsj}{\bar{R}_t(\bx, {\bf s}_j)}
\newcommand{\bff}{{\bf f}}

\newcommand{\bfu}{{\bf u}}
\newcommand{\bfp}{{\bf p}}
\newcommand{\bfs}{{\bf s}}

\newcommand{\cM}{\mathcal{M}}
\newcommand{\diver}{\mbox{div}}

\newcommand{\M}{{\mathcal M}}

\newcommand{\bV}{\mathbf{V}}
\newcommand{\bA}{\mathbf{A}}
\newcommand{\invt}{\frac{1}{t}}

\newcommand{\mathd}{\mathrm{d}}

 \newtheorem{theorem}{\textbf{Theorem}}[section]
 \newtheorem{lemma}{\textbf{Lemma}}[section]
 \newtheorem{remark}{\textbf{Remark}}[section]

 \newtheorem{assumption}{\textbf{Assumption}}[section]
\newtheorem{definition}{\textbf{Definition}}[section]

\newcommand{\R}{\mathbb{R}}

\numberwithin{equation}{section}

\makeatother

\begin{document}

\title{Convergence of the Point Integral method for the Poisson equation with Dirichlet boundary on point cloud}

% \author{
% Zuoqiang Shi}%
% \thanks{Yau Mathematical Sciences Center, Tsinghua University, Beijing, China,
% 100084. \textit{Email: zqshi@mail.tsinghua.edu.cn.}}%

% \author{ Jian Sun} %
% \thanks{Yau Mathematical Sciences Center, Tsinghua University, Beijing, China,
% 100084. \textit{Email: jsun@math.tsinghua.edu.cn.}}%
\author{
Zuoqiang Shi%
\thanks{Yau Mathematical Sciences Center, Tsinghua University, Beijing, China,
100084. \textit{Email: zqshi@mail.tsinghua.edu.cn.}%
}
\and Jian Sun %
\thanks{Yau Mathematical Sciences Center, Tsinghua University, Beijing, China,
100084. \textit{Email: jsun@math.tsinghua.edu.cn.}%
}
} 

% \keywords{point integral method; Laplace-Beltrami operator;
% Poisson equation; Dirichlet boundary;
% point clouds.}
\date{}
\maketitle
\begin{abstract}
The Poisson equation on manifolds plays an fundamental role in many applications. 
Recently, we proposed a novel numerical method called the Point Integral method (PIM)
to solve the Poisson equations on manifolds from point clouds. 
In this paper, we prove the convergence of the point integral method for solving the Poisson equation 
with the Dirichlet boundary condition. 
%Together with our previous result, 
%the solid theoretical foundation of the point integral method for Poisson equation on point cloud is 
%well established. 
\end{abstract}

% \vspace{0.1in}
% \noindent{\textbf{Keywords:}}
% point integral method; Laplace-Beltrami operator;
% Poisson equation; Dirichlet boundary;
% point clouds.

%\newpage

\section{Introduction}

In the past decades, machine learning attracts more and more attentions. In many problems of machine learning, data can be represented as 
a set of points in high dimensional Euclidean space, which is usually referred as point cloud. 
One fundamental problem in machine learning is to infer the value of a function on the whole point cloud from the value on a subset of the point cloud.
Harmonic function provides an efficient way to solve this problem. One need to find a harmonic function such that it coincides with the 
given value in the subset of the point cloud. Apparently, this harmonic function can be obtained by solving Laplace equation with Dirichlet type boundary condition.  

The partial differential equations on manifolds also arise in a wide variety of applications,  
including material science \cite{CFP97,EE08}, fluid flow \cite{GT09,JL04},
biology and biophysics \cite{BEM11,ES10,NMWI11,WD08}. In these problems, Dirichlet boundary condition is also very common. 

In 2D surfaces, people have developed many numerical methods to solve variety of PDEs,
such as 
surface finite element method \cite{DE-Acta}, level set method \cite{Bertalmio,XZ03}, grid based particle method \cite{LZ09,LZ11} 
and closest point method \cite{RM08,MR09}. These methods are difficult to solve PDEs on general point cloud in high dimensional space.

To discretize the differential operators on point cloud, several alternative numerical methods have been developed. 
Liang et al. proposed to discretize the differential operators on point cloud by local least square approximations of the manifold \cite{Liang13}.
 Later, Lai et al. proposed local mesh method to approximate the 
differential operators on point cloud \cite{Lai13}. The main idea is to approximate the manifold locally by polynomials or mesh. 
Once the local approximation is obtained, it is easy to discretize the differential operators. However, when the dimension of the 
manifold is high, the local approximation is not easy to construct.

In~\cite{LSS}, we proposed a novel numerical method, point integral method (PIM), 
to solve the Poisson equation on point cloud.
The main idea of the point integral method is to
approximate the Poisson equation by the following integral equation:
\begin{equation}
\label{eq:integral-intro}
  -\int_\M  \Delta_\M u(\by)\rhk d\mu_\by\approx \invt\int_{\M} \hk(u(\bx) - u(\by))\mathd\mu_\by-2\int_{\p\M}\rhk \frac{\p u}{\p\bn}(\by)\mathd \tau_\by,
\end{equation}
where $\bn$
is the out normal of $\M$, $\M$ is a smooth $k$-dimensional manifold embedded in $\mathbb{R}^d$ and $\p\M$ is the boundary of $\M$. 
$R_t(\bx,\by)$ and $\bar{R}_t(\bx,\by)$ are kernel functions given as follows
\begin{equation}
\label{eq:kernel}
R_t(\bx, \by) = C_tR\left(\frac{|\bx -\by|^2}{4t}\right),\quad
\bar{R}_t(\bx, \by) = C_t\bar{R}\left(\frac{|\bx -\by|^2}{4t}\right)
\end{equation}
where $C_t = \frac{1}{(4\pi t)^{k/2}}$ is the normalizing factor.
 $R\in C^2(\mathbb{R}^+) $ be a positive function which is integrable over $[0,+\infty)$,
\begin{equation*}
  \bar{R}(r)=\int_r^{+\infty}R(s)\mathd s.
\end{equation*}
$\Delta_\mathcal{M}=\diver(\nabla)$ is the Laplace-Beltrami operator on $\mathcal{M}$.  
Let $\Phi: \Omega\subset \mathbb{R}^k\rightarrow \cM\subset\mathbb{R}^d$ be a local parametrization of $\cM$ and $\theta\in \Omega$.
For any differentiable function $f:\cM\rightarrow \mathbb{R}$,
%let $F(\theta)=f(X(\theta))$,
 define the gradient on the manifold
\begin{align}
  \label{eq:diff-M}
  \nabla f(\Phi(\theta))&=\sum_{i,j=1}^m g^{ij}(\theta)\frac{\p \Phi}{\p\theta_i}(\theta)\frac{\p f(\Phi(\theta))}{\p\theta_j}(\theta),
\end{align}
and for vector field $F:\M\rightarrow T_\bx\M$ on $\M$, where $T_\bx\M$ is the tangent space of $\M$ at $\bx\in \M$, the divergence is defined as
\begin{align}
\label{eq:diver}
\diver (F)&= \frac{1}{\sqrt{\det G}}\sum_{k=1}^d\sum_{i,j=1}^m\frac{\p}{\p \theta_i}\left(\sqrt{\det G}g^{ij}F^k(\Phi(\theta))\frac{\p \Phi^k}{\p\theta_j}\right)
\end{align}
where $(g^{ij})_{i,j=1,\cdots,k}=G^{-1}$, $\det G$ is the determinant of matrix $G$ and $G(\theta)=(g_{ij})_{i,j=1,\cdots,k}$ is the first fundamental form which is defined by
\begin{eqnarray}
  \label{eq:remainn}
  g_{ij}(\theta)=\sum_{k=1}^d\frac{\p \Phi_k}{\p\theta_i}(\theta)\frac{\p \Phi_k}{\p\theta_j}(\theta),\quad i,j=1,\cdots,m.
\end{eqnarray}
and $(F^1(\bx),\cdots,F^d(\bx))^t$ is the representation of $F$ in the embedding coordinates.

Using the integral approximation, the Laplace-Beltrami operator is transfered to an integral operator. The integral operator is easy to
be discretized on point clouds.
Similar idea is also used in nonlocal diffusion and peridynamic model \cite{Du-SIAM,book-nonlocal,DGLZ13,DJTZ13,ZD10}.

In this paper, we focus on the Dirichlet problem for the Poisson equation on 
a smooth, compact $k$-dimensional submanifold $\mathcal{M}$ 
in $\R^d$. 
\begin{equation} 
\left\{\begin{array}{rl}
      -\Delta_\mathcal{M} u(\bx)=f(\bx),&\bx\in \mathcal{M} \\
      u(\bx)=b(\bx),& \bx\in \p \mathcal{M}
\end{array}
\right.  
\label{eqn:dirichlet-intro} 
\end{equation}
The integral approximation does not apply on the Dirichlet problem directly, since the normal derivative is required in the integral approximation while it is 
not given in Dirichlet problem. To solve this problem, we use the Robin problem to approximate the original Dirichlet problem.
\begin{equation} 
\left\{\begin{array}{rl}
      -\Delta_\mathcal{M} u(\bx)=f(\bx),&\bx\in \mathcal{M} \\
      u(\bx)+\beta\frac{\p u}{\p \bn}=b(\bx),& \bx\in \p \mathcal{M}, 
\end{array}
\right. 
\label{eqn:robin-intro} 
\end{equation}
Above Robin problem approximates the Dirichlet problem~\eqref{eqn:dirichlet-intro} when the parameter $\beta$ is small.
At the same time, it can be approximated by following integral equation
\begin{align}
\label{eqn:integral_dirichlet_intro}
  \frac{1}{t}\int_\M R_t(\bx,\by)(u(\bx)-u(\by))\mathd\mu_\by-\frac{2}{\beta}\int_{\p\M}&\rhk \left(b(\by)-u(\by)\right)\mathd \tau_\by\\
&=\int_\M \rhk f(\by)d\mu_\by.\nonumber
\end{align}
After discretizing this integral equation on point cloud, we get a numerical scheme to solve the Dirichlet problem \eqref{eqn:dirichlet-intro}.
The detailed algorithm is given in Section 2.

The main contribution of this paper is that, for Poisson equation with Dirichlet boundary condition,  we prove that the numerical solution computed 
by the PIM converges to the exact solution in $H^1$ norm as the point cloud converges to the underlying smooth manifold. 
In \cite{SS-neumann}, the convergence of the point integral method for Neumann problem has been proved. The method used in this paper is similar as that 
in \cite{SS-neumann}. The main difference is that in Dirichlet problem, we need to consider the effect of the boundary term which introduce more difficulties 
in the analysis.

% It is well known in the numerical analysis that the convergence is the summation of consistency and stability. 
% We prove that the coercivity of the original Laplace-Beltrami operator is preserved in the point integral method. This 
% implies the stability of the point integral method. Together with the estimate of the truncation error, we get the convergence of the point integral method.
% We will show that the PIM is consistent (Theorem \ref{thm:integral_error} and \ref{thm:dis_error}) and 
% stable (Theorem~\ref{thm:regularity} and \ref{thm:regularity_boundary}).
% In proving the stability of the PIM, the key part is that that the point integral method preserves 
% the coercivity of the Laplace-Beltrami operator (see Theorem \ref{thm:elliptic_v} and \ref{thm:elliptic_L_t}), 
% which is interesting on its own. 

The remaining of this paper is organized as following. In Section 2, 
we describe the point integral method for Poisson equation with Dirichlet boundary condition. %a brief introduction of the Point Integral method. 
The convergence result is stated in Section 3. The structure of the proof is shown in Section 4. 
% In section 6, we show several basic estimates related to the properties of smooth submanifolds 
% and the convolutions with the kernel $R$. 
%which will be used often in the proofs. 
The main body of the proof is in Section 5, Section 6. 
Finally, conclusions and discussion on the future work are given in Section 7. 

\section{Point integral method}
% This is a continuation of~\cite{SS-neumann} in which the convergence of the PIM for
% solving the Poisson equation with the Neumann boundary is proven.
% The purpose of this paper is to prove the convergence of the PIM for the Dirichlet
% boundary. In particular,
In this paper, we consider the Dirichlet problem for the Poisson equation on 
a smooth, compact $k$-dimensional submanifold $\mathcal{M}$ 
in $\R^d$. 
\begin{equation} 
\left\{\begin{array}{rl}
      -\Delta_\mathcal{M} u(\bx)=f(\bx),&\bx\in \mathcal{M} \\
      u(\bx)=b(\bx),& \bx\in \p \mathcal{M}
\end{array}
\right.  
\label{eqn:dirichlet} 
\end{equation}
%and Dirichlet boundary condition, 
%\begin{eqnarray}
%\label{eq-dirichelet}
  %\left\{\begin{array}{rl}
      %-\Delta_\mathcal{M} u(\bx)=f(\bx),&\bx\in \mathcal{M} \\
      %u(\bx)=b(\bx),&  \bx\in \p \mathcal{M}
%\end{array}\right.
%\end{eqnarray}
where $\Delta_\mathcal{M}$ is the Laplace-Beltrami operator on $\mathcal{M}$ which has been defined in previous section.  

Based on the integral approximation \eqref{eq:integral-intro}, the Dirichlet problem \eqref{eqn:dirichlet} is well approximated by 
an integral equation,
\begin{equation}
  L_tu(\bx)-2\int_{\p\M}\rhk \frac{\p u}{\p\bn}(\by)\mathd \tau_\by=\int_\M \rhk f(\by)d\mu_\by, 
%\tag{N2.a}
\label{eqn:integral}
\end{equation}
where $\bn$ is the out normal of $\M$, $R_t(\bx,\by)$ and $\bar{R}_t(\bx,\by)$ are kernel functions 
given in \eqref{eq:kernel}, $L_t$ is an integral operator defined as 
\begin{equation}
\label{eq:Lt}
  L_tu(\bx)=\frac{1}{t}\int_\M R_t(\bx,\by)(u(\bx)-u(\by))\mathd\mu_\by, 
\end{equation}
In the integral equation \eqref{eqn:integral}, the Neumann boundary is natrual. It does not apply on the Dirichlet problem directly, 
since the normal derivative $\frac{\p u}{\p\bn}$
is not given in the Dirichlet problem. To enforce the Dirichlet boundary, 
we use the Robin boundary to bridge the Neumann boundary and the Dirichlet boundary. 
In particular, we consider the following Robin problem
\begin{equation} 
\left\{\begin{array}{rl}
      -\Delta_\mathcal{M} u(\bx)=f(\bx),&\bx\in \mathcal{M} \\
      u(\bx)+\beta\frac{\p u}{\p \bn}=b(\bx),& \bx\in \p \mathcal{M}, 
\end{array}
\right.  
\label{eqn:robin} 
\end{equation}
The Robin problem approximates the Dirichlet problem~\eqref{eqn:dirichlet} when the parameter $\beta$ is small.
On the other hand, it can be approximated by following integral equation
\begin{equation}
  L_tu(\bx)-\frac{2}{\beta}\int_{\p\M}\rhk \left(b(\by)-u(\by)\right)\mathd \tau_\by=\int_\M \rhk f(\by)d\mu_\by
\label{eqn:integral_dirichlet}
\end{equation}
when the parameter $t$ is small.  % This integral equation is obtained from the equation~\eqref{eqn:integral} 
% by substituting $\frac{\p u}{\p \bn}$ with $((b(\by)-u(\by))/\beta$. 

In the PIM, we assume a set of points $P$ samples the submanifold $\M$ and a subset $S\subset P$ samples the boundary of $\M$.
List the points in $P$ respectively $S$ in a fixed
order $P=(\bfp_1, \cdots, \bfp_n)$ where $\bfp_i \in \R^d, 1\leq i\leq n$, respectively $S=(\bfs_1, \cdots, \bfs_m)$ where $\bfs_i \in P$.
In addition, assume we are given
two vectors $\bV = (V_1, \cdots, V_n)$ where $V_i$ is an volume weight of $\bfp_i$ in $\M$, and
$\bA= (A_1, \cdots, A_m)$ where $A_i$ is an area weight of $\bfs_i$ in $\p \M$.% , so that for any Lipschitz
% function $f$ on $\M$ respectively $\p \M$, $\int_\M f(\bx) d\mu_\bx$ respectively $\int_{\p \M} f(\bx) d\tau_\bx$ can be approximated by
% $\sum_{i=1}^n f(\bfp_i) V_i$ respectively $\sum_{i=1}^m f(\bfs_i) A_i$. Here $d\mu_\bx$ and $d\tau_\bx$ are the volume form
% of $\M$ and $\p \M$, respectively. 

The integral equation \eqref{eqn:integral_dirichlet} is easy to be discretized over the point cloud 
 $(P,S,\bV,\bA)$ to obtain the following linear system of $\bfu = (u_1, \cdots, u_n)$,
\begin{equation}
\mathcal{L} \bfu (\bfp_i) - \frac{2}{\beta}\sum_{\bfs_j \in S} \rhkpisj (b(\bfs_j)-u(\bfs_j))A_j = 
\sum_{\bfp_j \in P} \rhkpipj f(\bfp_j)V_j.
\label{eqn:dis_dirichlet}
\end{equation}
where
 \begin{equation}
\label{eq:Lt-dis}
 \mathcal{L} \bfu (\bfp_i) = \invt\sum_{\bfp_j \in P} \hkpipj(u_i - u_j)V_j
 \end{equation}
is the discrete Laplace operator.

The purpose of this paper is to show that the solution of the linear system \eqref{eqn:dis_dirichlet} converges to the 
solution of the Dirichlet problem \eqref{eqn:dirichlet} as the point cloud $P$ converges to the underlying manifold $\M$ and $t,\beta$ go to 0.
The idea to prove the convergence is similar as that in \cite{SS-neumann}. The detailed analysis will be given in the subsequent sections.

% Fortunately, the integral Laplace operator $L_t$ is stable enough and able to control the perturbation 
% introduced by this extra integral over the boundary 
%(see Theorem~\ref{thm:regularity_robin} and \ref{thm:regularity_boundary_robin}). 

% \subsection{Related work}
% Much of research has been done on solving PDEs on manifolds. The interested
% readers are referred to~\cite{SS-neumann} and the references therein for the 
% related work. Here we emphysize one point related to the harmonic function, i.e., 
% the solution to the Laplace equation $\Delta_\M u = 0$.  
% The discrete harmonicity $\mathcal{L} \bfu = 0$ is extensively studied in the 
% graph theory~\cite{Chung}, and is closed related to random walks and electric networks
% on graphs~\cite{Doyle1984}. Comparing the equation ~\eqref{eqn:dirichlet} to 
% the equation ~\eqref{eqn:dis_dirichlet}, we notice that the smooth harmonicity 
% $\Delta_\M u = 0$ may not be well approximated by the discrete harmonicity 
% $\mathcal{L} \bfu = 0$  as there is an extra integral over the boundary.  
% Du et al.~\cite{Du-SIAM} also noticed this phenomenon in their study of 
% nonlocal diffusion problems where the nonlocal operator takes the same form 
% as the integral Laplace operator $L_t$, and proposed the so-called volume
% constraints to enforce the Dirichlet boundary. In~\cite{S15}, 
% it is shown that this volume constraint approach also produces a convergent solution
% to the problem~\eqref{eqn:dirichlet}. 

\section{Assumptions and Results}

The main contribution in this paper is to establish the convergence results for the point integral method
for solving the problem~\eqref{eqn:dirichlet}.
To simplify the notation and make the proof concise, we consider the homogeneous Dirichlet boundary
conditions, i.e. 
\begin{equation} 
\left\{\begin{array}{rl}
      -\Delta_\mathcal{M} u(\bx)=f(\bx),&\bx\in \mathcal{M} \\
      u(\bx)=0,& \bx\in \p \mathcal{M}
\end{array}
\right.  
\label{eqn:dirichlet-homo} 
\end{equation}
 The analysis can be easily generalized to 
the non-homogeneous boundary conditions. 

The corresponding numerical scheme is 
\begin{equation}
\invt\sum_{\bfp_j \in P} R_t(\bfp_i,\bfp_j)(u(\bfp_i) - u(\bfp_j))V_j 
+\frac{2}{\beta}\sum_{\bfs_j\in S}\bar{R}_t(\bfp_i,\bfs_j)u(\bfs_j)A_j = \sum_{\bfp_j \in P} \bar{R}_t(\bfp_i,\bfp_j) f_jV_j.
\label{eqn:dis-homo}
\end{equation}
where $f_j=f(\bfp_j)$.

Before proving the convergence of the point integral method, we need to clarify the meaning of the convergence between the point cloud $(P,S,\mathbf{V},\mathbf{A})$
and the manifold $\M$. In this paper, we consider the convergence in the sense that $$h(P, S, \mathbf{V}, \mathbf{A}, \M, \p\M)\rightarrow 0$$ 
where $h(P, S, \mathbf{V}, \mathbf{A}, \M, \p\M)$ is the 
{\it integral accuracy index} defined as following,
\begin{definition}[Integral Accuracy Index]
  \label{def:h}
For the point cloud $(P,S,\mathbf{V},\mathbf{A})$ which samples the manifold $\M$ and $\p\M$, 
the integral accuracy index $h(P, S, \mathbf{V}, \mathbf{A}, \M, \p\M)$ is defined as
\begin{align*}
  h(P, S, \mathbf{V}, \mathbf{A}, \M, \p\M)=\max\left\{h(P,\mathbf{V},\M),h(S,\mathbf{A},\p\M)\right\}
\end{align*}
and
\begin{align*}
h(P,\mathbf{V},\M)=\sup_{f\in C^1(\M)}\frac{\left|\int_\M f(\by) \mathd\mu_\by - \sum_{\bfp_i\in P} f(\bfp_i)V_i\right|}{|\text{\rm supp}(f)|\|f\|_{C^1(\M)}},\\
h(S,\mathbf{A},\p\M)=\sup_{g\in C^1(\p\M)}\frac{\left|\int_{\p\M} g(\by) \mathd\tau_{\by} - \sum_{\bfs_i\in S} g(\bfs_i)A_i\right|}{|\text{\rm supp}(g)|\|g\|_{C^1(\p\M)}}
\end{align*}
% where $\|f\|_{C^1(\M)} = \|f\|_\infty +\|\nabla f\|_\infty$ and
% $|\text{\rm supp}(f)|$ is the volume of the support of $f$.
To simplify the notation, we denote $h=h(P, S, \mathbf{V}, \mathbf{A}, \M, \p\M)$ in the rest of the paper.
\end{definition}

Using the definition of integrable index, we say that the point cloud $(P,S,\mathbf{V},\mathbf{A})$ converges to the manifold $\M$ if 
$h\rightarrow 0$. The convergence analysis in this paper is based on the assumption that $h$ is small enough.
% \begin{remark}
%   In some sense, $h(P,\mathbf{V},\M)$ is a measure of the density of the point cloud. If the point cloud is uniformly distributed on the manifold, 
% from central limit theorem, $h(P,\mathbf{V},\M)\sim O(1/\sqrt{n})$ where $n$ is the number of point in $P$.
% \end{remark}

To get the convergence, we also need some assumptions on the regularity of the submanifold $\M$ and
the integral kernel function $R$.
\begin{assumption}
\label{assumptions}
\begin{itemize}
\item[]
\item \rm Smoothness of the manifold: $\M, \p\M$ are both compact and $C^\infty$ smooth $k$-dimensional submanifolds isometrically embedded in a Euclidean space $\mathbb{R}^d$.
\item \rm Assumptions on the kernel function $R(r)$:
\begin{itemize}
\item[\rm (a)] \rm Smoothness: $R\in C^2(\mathbb{R}^+)$;
\item[(b)] Nonnegativity: $R(r)\ge 0$ for any $r\ge 0$.
\item[(c)] Compact support:
%$R(r) \le 1$ for $\forall r \in \R^+$ and $R(r) = 0$ for $\forall r >1$.
$R(r) = 0$ for $\forall r >1$;
\item[(d)] Nondegeneracy:
 $\exists \delta_0>0$ so that $R(r)\ge\delta_0$ for $0\le r\le\frac{1}{2}$.
\end{itemize}
%\item Assumptions on $t$ and $h$: $t$ and $h/\sqrt{t}$ are small enough, i.e., $t$ and $h/\sqrt{t}$ are less than a positive constant
%which only depends on $\M$.
\end{itemize}
\end{assumption}
\begin{remark}
  The assumption on the kernel function is very mild. The compact support assumption can be relaxed to exponentially decay, like Gaussian kernel.
 In the nondegeneracy assumption, $1/2$ may be replaced by a positive number $\theta_0$ with $0<\theta_0<1$. Similar assumptions on the kernel function 
is also used in analysis the nonlocal diffusion problem \cite{DLZ13}.
\end{remark}

% \begin{remark}
% We remark that after some mild modifications
% of the proof,
% the same convergence results also hold for any kernel function $R$ that decays exponentially, like the Gaussian kernel 
% $G_t(\bx, \by) = C_t\exp\left(-\frac{|\bx-\by|^2}{4t}\right)$.  
% \end{remark}

All the analysis in this paper is under the assumptions in Assumption \ref{assumptions} and $h$, $t$ are small enough. 
In the theorems and the proof, without introducing any confusions, we omit the statement of the assumptions.

To compare the discrete numerical solution with the continuous exact solution, we 
interpolate the discrete solution $\bfu = (u_1, \cdots, u_n)$ of
the problem~\eqref{eqn:dis_dirichlet} onto the smooth manifold using following interpolation operator: 
\begin{equation}
\label{eqn:interp_dirichlet}
I_{\bff}(\bfu) (\bx) = \frac{ \sum_{\bfp_j \in P} \hkxpj u_j V_j - \frac{2t}{\beta} \sum_{\bs_j \in S} \rhkxsj u_jA_j + t\sum_{\bfp_j \in P} \rhkxpj f_jV_j} {\sum_{\bfp_j \in P} \hkxpj V_j }.
\end{equation}
where $\bff=[f_1,\cdots,f_n]=[f(\bfp_1),\cdots,f(\bfp_n)]$. 
It is easy to verify that $I_{\bff}(\bfu)$ interpolates $\bfu$ at the sample points $P$, i.e., 
$I_{\bff}(\bfu)(\bfp_j) = u_j$ for any $j$. In the  
analysis, $I_{\bff}(\bfu)$ is used as the numerical solution of \eqref{eqn:dirichlet} instead of the discrete solution $\bfu$.

 Now, we can state the main result. 

\begin{theorem}
 Let
$u$ is the solution to Problem~\eqref{eqn:dirichlet-homo} with  $f\in C^1(\M)$. 
Set ${\bf f} = (f(\bfp_1), \cdots, f(\bfp_n))$. If 
the vector $\bfu$ is the solution to the problem~\eqref{eqn:dis-homo}. 
There exists constants $C$, $T_0$ and $r_0$ only depend on 
$\M$ and $\p\M$,
so that for any $t\le T_0$,
\begin{equation}
  \|u-I_{\bff}(\bfu)\|_{H^1(\M)} \leq C\left(\frac{h}{t^{3/2}}+t^{1/2}+\beta^{1/2}\right)\|f\|_{C^1(\M)}. 
\end{equation}
as long as $\frac{h}{t^{3/2}}\le r_0$ and $\frac{\sqrt{t}}{\beta}\le r_0$. %Here $s=k/2+2$.
%In particular, if set $t=h^{4/9}$,  we obtain a convergence rate $h^{1/9}$. 
\label{thm:poisson_dirichlet}
\end{theorem}

%\vspace{2mm}
%\noindent{\bf Organization of the paper.~}
%The remaining of this paper is organized as following. In Section 2, we give the structure of the proof. 
%In section 3, we show several basis estimates related to the property of the manifold and the integral operator
%which will be used often in the proofs. 
%The main body of the proof is in Section 4 and Section 5. Some discussion and 
%conclusions are made in Section 6.

\section{Structure of the Proof}
\label{sec:intermediate}

In the point integral method, we use Robin boundary problem \eqref{eqn:robin-homo} to approximate the Dirichlet boundary problem \eqref{eqn:dirichlet-homo}. 
First, we show that the solution of the Robin problem converges to the solution of the Dirichlet problem as the parameter $\beta\rightarrow 0$.
\begin{theorem} Suppose $u$ is the solution of the Dirichlet problem \eqref{eqn:dirichlet-homo} and 
$u_{\beta}$ is the solution of the Robin problem 
\begin{equation} 
\left\{\begin{array}{rl}
      -\Delta_\mathcal{M} u(\bx)=f(\bx),&\bx\in \mathcal{M} \\
      u(\bx)+\beta \frac{\p u}{\p \bn}(\bx)=0,& \bx\in \p \mathcal{M}
\end{array}
\right.  
\label{eqn:robin-homo} 
\end{equation}
then 
  \begin{eqnarray}
    \|u-u_{\beta}\|_{H^1(\M)}\le C \beta^{1/2}\|f\|_{L^2(\M)}.\nonumber
  \end{eqnarray}
\label{thm:diff_D2R}
\end{theorem}
\begin{proof}
  Let $w=u-u_{\beta}$, then $w$ satisfies
  \begin{eqnarray}
    \left\{\begin{array}{rl}
\Delta_\M w=0,& \mbox{on}\;\M,\\
w+\beta \frac{\p w}{\p \bn}=\beta\frac{\p u}{\p \bn},& \mbox{on}\;\p\M.
\end{array}\right.\nonumber
  \end{eqnarray}
By multiplying $w$ on both sides of the equation and integrating by parts, we can get
  \begin{eqnarray}
    0&=&\int_\M w\Delta_\M w\mathd \mu_\bx\nonumber\\
&=&-\int_\M |\nabla w|^2\mathd \mu_\bx+\int_{\p\M} w\frac{\p w}{\p \bn}\mathd\tau_\bx\nonumber\\
&=&-\int_\M |\nabla w|^2\mathd \mu_\bx-\frac{1}{\beta}\int_{\p\M} w^2\mathd\tau_\bx
+\int_{\p\M} w \frac{\p u}{\p \bn}\mathd\tau_\bx\nonumber\\
&\le & -\int_\M |\nabla w|^2\mathd \mu_\bx-\frac{1}{2\beta}\int_{\p\M} w^2\mathd\tau_\bx
+2\beta \int_{\p\M} \left|\frac{\p u}{\p \bn}\right|^2\mathd\tau_\bx,\nonumber
  \end{eqnarray}
which implies that
\begin{eqnarray}
  \int_\M |\nabla w|^2\mathd \mu_\bx+\frac{1}{2\beta}\int_{\p\M} w^2\mathd\tau_\bx
\le 2\beta \int_{\p\M} \left|\frac{\p u}{\p \bn}\right|^2\mathd\tau_\bx.\nonumber
\end{eqnarray}
Moreover, we have
\begin{eqnarray}
  \|w\|_{L^2(\M)}^2\le C\left(\int_\M |\nabla w|^2\mathd \mu_\bx+\frac{1}{2\beta}\int_{\p\M} w^2\mathd\tau_\bx\right)
\le C\beta \int_{\p\M} \left|\frac{\p u}{\p \bn}\right|^2\mathd\tau_\bx.\nonumber
\end{eqnarray}
Combining above two inequalities and using the trace theorem, we get
\begin{eqnarray}
  \|u-u_{\beta}\|_{H^1(\M)}\le C\beta^{1/2}\left\|\frac{\p u}{\p \bn}\right\|_{L^2(\p\M)} \le C\beta^{1/2}\|u\|_{H^2(\M)}.\nonumber
\end{eqnarray}
The proof is complete using that
\begin{equation*}
  \|u\|_{H^2(\M)}\le C\|f\|_{L^2(\M)}.
\end{equation*}
\end{proof}

%The following theorem says that the operator $K_t$ well approximates the Laplace-Beltrami
%operator $\Delta_\M$.
Next, we prove the solution of \eqref{eqn:dis-homo} converges to the solution of the Robin prolem \eqref{eqn:robin-homo} as
$h$, $t$ go to 0.
Comparing to the Neumann boundary problem considered in~\cite{SS-neumann}, in \eqref{eqn:dis-homo}, the unknown variables $u_i$
not only appear in the discrete Laplace operator $L_t$, but also appear in an integral over the boundary. 
Therefore, instead of showing the stability for the integral Laplace operator $L_t$ as in~\cite{SS-neumann}, 
we need to consider the stability for the following integral operator
\begin{eqnarray}
K_{t}u(\bx) &=& \invt\int_\M \hk(u(\bx) - u(\by))\mathd\mu_\by+\frac{2}{\beta}\int_{\p\M}\rhk u(\by)\mathd \tau_\by.
\label{eqn:laplace_integral_dirichlet}
\end{eqnarray}
This is the most difficult part in this paper.

\begin{theorem}
Let $u(\bx)$ solves following equation with $r\in H^1(\M)$
\begin{eqnarray}
  K_t u = r.\nonumber
\end{eqnarray}
Then, there exist constants $C, T_0, r_0>0$ independent on $t$, such that
\begin{eqnarray}
  \|u\|_{H^1(\M)}\le C\left(\|r\|_{L^2(\M)}+\frac{t}{\sqrt{\beta}}\| r\|_{H^1(\M)}\right),\nonumber
\end{eqnarray}
as long as $t\le T_0$ and $\frac{\sqrt{t}}{\beta}\le r_0$.
\label{thm:regularity_robin}
\end{theorem}

To apply the stability result, we need $L_2$ estimate of $K_t(u_\beta-I_{\bff}(\bfu))$ and $\nabla K_t(u_\beta-I_{\bff}(\bfu))$.
In the analysis, the truncation error $K_t(u_\beta-I_{\bff}(\bfu))$ is further splitted to two terms
\begin{equation*}
  K_t(u_\beta-I_{\bff}(\bfu))=K_t(u_\beta-u_{\beta,t}))+K_t(u_{\beta,t}-I_{\bff}(\bfu))
\end{equation*}
where $u_{\beta,t}$ is the solution of the integral equation 
\begin{equation}
  \label{eq:integral-homo}
\invt\int_{\M} \hk(u(\bx) - u(\by))\mathd\mu_\by+\frac{2}{\beta}\int_{\p\M}\rhk u(\by)\mathd \tau_\by=\int_\M f(\by)\bar{R}_t(\bx,\by)\mathd \mu_\by.
\end{equation}
The first term $K_t(u_\beta-u_{\beta,t})$ is same as that in the Neumann boundary problem \cite{SS-neumann}. It also has boundary layer structure. 
\begin{theorem}
% Under the assumptions in Section \ref{assumptions},
 Let $u(\bx)$ be the solution
of the problem~\eqref{eqn:dirichlet-homo} and $u_t(\bx)$ be the solution of the corresponding
integral equation \eqref{eq:integral-homo}. Let
\begin{align}
\label{eq:error_boundary}
&I_{bd} =\sum_{j=1}^d \int_{\p\M}n^j(\by)(\bx-\by)\cdot\nabla(\nabla^ju(\by))\rhk p(\by)\mathd \tau_\by,
\end{align}
and
\begin{align}
K_t (u- u_t)=I_{in}+I_{bd}\nonumber.
\end{align}
where $\bn(\by)=(n^1(\by),\cdots,n^d(\by))$ is the out normal vector of $\p\M$ at $\by$, $\nabla^j$ is the $j$th component of gradient $\nabla$.
%$\al=X^{-1}(\bx), \xi(\bx,\by)=X^{-1}(\bx)-X^{-1}(\by)$ and $\eta(\bx,\by)=\xi^{i'}(\bx,\by)\p_{i'}\Phi(\al)$.

If $u\in H^3(\M)$, then there exists constants
$C, T_0$ depending only on $\M$ and $p(\bx)$, so that,
\begin{eqnarray}
\label{eqn:integral_error_int}
\left\|I_{in}\right\|_{L^2(\M)}\le Ct^{1/2}\|u\|_{H^3(\mathcal{M})},\quad
\left\|\nabla I_{in}\right\|_{L^2(\M)} \leq C\|u\|_{H^3(\mathcal{M})},
\end{eqnarray}
as long as $t\le T_0$.
\label{thm:integral_error_robin}
\end{theorem}

The estimate of the second term, $K_t(u_{\beta,t}-I_{\bff}(\bfu))$, is given in following theorem.
\begin{theorem}
%Under the assumptions in Section \ref{assumptions},
Let $u_t(\bx)$ be the solution of the problem~~\eqref{eq:integral-homo} and $\bfu$ be the solution of
the problem~\eqref{eqn:dis-homo}. If
$f\in C^1(\M)$ ,
%$f\in C(\M), g\in C(\p\M)$
 then there exists constants
$C, T_0$ depending only on $\M$, so that
\begin{eqnarray}
\|K_{t} \left(I_{\bff}\bfu - u_{t}\right)\|_{L^2(\M)} &\leq& \frac{Ch}{t^{3/2}}\|f\|_{C^1(\M)},\\
\label{eqn:dis_error_l2}
\|\nabla K_{t} \left(I_{\bff}\bfu - u_{t}\right)\|_{L^2(\M)} &\leq& \frac{Ch}{t^{2}}\|f\|_{C^1(\M)}.
\label{eqn:dis_error_dl2}
\end{eqnarray}
as long as $t\le T_0$ and $\frac{h}{\sqrt{t}}\le T_0$.%, $h(P,\mathbf{V},\M)$ is the integral difference index in 
%Definition \ref{def:h}.
\label{thm:dis_error_robin}
\end{theorem}
% Using the definition of the boundary term $I_{bd}$, \eqref{eq:error_boundary},
% it is easy to check that 
% \begin{eqnarray*}
% \left\|I_{bd}\right\|_{L^2(\M)}= O(t^{1/4}),\quad
% \left\|\nabla I_{bd}\right\|_{L^2(\M)} = O(t^{-1/2}),
% \end{eqnarray*}
% Based on this estimation, Theorem \ref{thm:regularity} and Theorem \ref{thm:integral_error} give that 
% \begin{equation*}
%   \|u-u_t\|_{H^1(\M)}=O(t^{1/4}).
% \end{equation*}
% This proves the convergence, however the convergence rate is relatively low. This low rate comes from the boundary term. From interior term only, 
% the rate is $\sqrt{t}$. Notice that the boundary term has a specific integral formula given in \eqref{eq:error_boundary}. 
% Using this formula, we know that the boundary term concentrates in a small layer adjacent to the boundary whose width is of the order of $\sqrt{t}$ and vanish in the 
% interior region.
% Utilizing this special structure, we could get better convergence rate with the help of a stability estimate specifically for the boundary term, which is 
% given in Theorem \ref{thm:regularity_boundary}.

Corresponding to the boundary layer structure in Theorem \ref{thm:integral_error_robin}, we need stability of $K_t$ for the boundary term.
\begin{theorem}
Let $u(\bx)$ solves the integral equation
\begin{eqnarray*}
  K_t u (\bx)= \int_{\p\M}\mathbf{b}(\by)\cdot(\bx-\by)\rhk \mathd \tau_\by.
\end{eqnarray*}
There exist constant $C>0, T_0>0$ independent on $t$, such that
\begin{eqnarray*}
  \|u\|_{H^1(\M)}\le C\sqrt{t}\;\|\mathbf{b}\|_{H^1(\M)}.
\end{eqnarray*}
as long as $t\le T_0$.
\label{thm:regularity_boundary_robin}
\end{theorem}

Theorem~\ref{thm:poisson_dirichlet} is an easy corollary from Theorems \ref{thm:diff_D2R}, Theorems \ref{thm:regularity_robin}, \ref{thm:dis_error_robin}, 
\ref{thm:integral_error_robin} and \ref{thm:regularity_boundary_robin}. The detailed proof is omitted here. 
% Theorem \ref{thm:dis_error_robin} and Theorem \ref{thm:regularity_robin} imply that 
%  $$\|u_t-I_{\bff}(\bfu)\|_{H^1(\M)}=O\left(\frac{h(P,\mathbf{V},\M)}{t^{3/2}}\right).$$ 
% and Theorem \ref{thm:regularity_robin}, \ref{thm:integral_error_robin} and \ref{thm:regularity_boundary_robin} imply
%  $$\|u-u_t\|_{H^1(\M)}=O\left(t^{1/2}\right),$$ 
% which prove Theorem~\ref{thm:poisson_dirichlet}. 

Proof of Theorem \ref{thm:integral_error_robin} is essentially same as the proof of Theorem 4.3 in \cite{SS-neumann}.
In the rest of the paper, we prove Theorem 
  \ref{thm:regularity_robin}, \ref{thm:dis_error_robin} and \ref{thm:regularity_boundary_robin} respectively.

\section{Stability of $K_t$ (Theorem \ref{thm:regularity_robin} and \ref{thm:regularity_boundary_robin})}
\label{sec:stability}
In this section, we will prove Theorem \ref{thm:regularity_robin} and \ref{thm:regularity_boundary_robin}. 
Both these two theorems are concerned with the stability of $K_t$, which are essential in the convergence analysis. 

In the proof, we need following theorem which has been proved in \cite{SS-neumann}. 
\begin{theorem} For any function $u\in L^2(\mathcal{M})$, 
there exists a constant $C>0$ independent on $t$ and $u$, such that
  \begin{eqnarray*}
    \left<u,L_t u\right>_\M \ge C\int_\mathcal{M} |\nabla v|^2\mathd \mu_\bx
  \end{eqnarray*}
where $\left<f, g\right>_{\mathcal{M}} = \int_{\mathcal{M}} f(\bx)g(\bx)\mathd\mu_\bx$ for any $f, g\in L_2(\mathcal{M})$, and
\begin{eqnarray}
v(\bx)=\frac{C_t}{w_t(\bx)}\int_{\mathcal{M}}R\left(\frac{|\bx-\by|^2}{4t}\right) u(\by)\mathd \mu_\by,
\label{eqn:smooth_v}
\end{eqnarray}
and $w_t(\bx) = C_t\int_{\mathcal{M}}R\left(\frac{|\bx-\by|^2}{4t}\right)\mathd \mu_\by$.
\label{thm:elliptic_v}
\end{theorem}

\subsection{Stability of $K_t$ for interior term (Theorem \ref{thm:regularity_robin})}

Using Theorem \ref{thm:elliptic_v}, we have
  \begin{equation}
    \|\nabla v\|_{L^2(\M)}^2\le C \left<u,L_tu\right>=\int_\M u(\bx)r(\bx)\mathd\mu_\bx - \frac{2}{\beta}\int_\M u(\bx)\left(\int_{\p\M}\bar{R}_t(\bx,\by)u(\by)\mathd\tau_\by\right)
\mathd\mu_\bx.
\label{ineq:l2_dv}
  \end{equation}
where $v$ is the same as defined in Theorem \ref{thm:elliptic_v}.
We control the second term on the right hand side of \eqref{ineq:l2_dv} as follows. 
\begin{eqnarray}
   && \left|\int_\M u(\bx)\left(\int_{\p\M}\left(\bar{R}_t(\bx,\by)-\frac{\bar{w}_t(\by)}{w_t(\by)}R_t(\bx,\by)\right)u(\by)\mathd\tau_\by\right)\mathd\mu_\bx\right|\nonumber\\
&= & \left|\int_{\p\M} u(\by)\left(\int_{\M}\left(\bar{R}_t(\bx,\by)-\frac{\bar{w}_t(\by)}{w_t(\by)}R_t(\bx,\by)\right)u(\bx)\mathd\mu_\bx\right)\mathd\tau_\by\right|\nonumber\\  
&= & \left|\int_{\p\M} \frac{1}{w_t(\by)}u(\by)
\left(\int_{\M}\left(w_t(\by)\bar{R}_t(\bx,\by)-\bar{w}_t(\by)R_t(\bx,\by)\right)u(\bx)\mathd\mu_\bx\right)\mathd\tau_\by\right|\nonumber\\  
&\le & C\|u\|_{L^2(\p\M)}\left(\int_{\p\M}\left(\int_{\M}\left(w_t(\by)\bar{R}_t(\bx,\by)-\bar{w}_t(\by)R_t(\bx,\by)\right)u(\bx)\mathd\mu_\bx\right)^2\mathd\tau_\by\right)^{1/2}, 
\nonumber
  \end{eqnarray}
where $\bar{w}_t(\bx) = \int_\M \bar{R}_t(\bx, \by) \mathd \mu_\by$.
Noticing that
  \begin{eqnarray}
&&   \int_{\M}\left(w_t(\by)\bar{R}_t(\bx,\by)-\bar{w}_t(\by)R_t(\bx,\by)\right)u(\bx)\mathd\mu_\bx\nonumber\\
&=& \int_{\M}\int_\M R_t(\by,\mathbf{z})\bar{R}_t(\bx,\by)\left(u(\bx)-u(\B{z})\right)\mathd\mu_\bx\mathd \mu_{\B{z}},
\nonumber
  \end{eqnarray}
we have
  \begin{eqnarray}
&&\int_{\p\M}\left(\int_{\M}\left(w_t(\by)\bar{R}_t(\bx,\by)-\bar{w}_t(\by)R_t(\bx,\by)\right)u(\bx)\mathd\mu_\bx\right)^2\mathd\tau_\by\nonumber\\
&\le&    \int_{\p\M}\left(\int_{\M}\int_\M R_t(\by,\mathbf{z})\bar{R}_t(\bx,\by)\left(u(\bx)-u(\B{z})\right)\mathd\mu_\bx\mathd \mu_{\B{z}}\right)^2\mathd\tau_\by
\nonumber\\
&\le &  \int_{\p\M}\left(\int_{\M}\int_\M R_t(\by,\mathbf{z})\bar{R}_t(\bx,\by)\mathd\mu_\bx\mathd \mu_{\B{z}}\right)\left(\int_{\M}\int_\M R_t(\by,\mathbf{z})\bar{R}_t(\bx,\by)\left(u(\bx)-u(\B{z})\right)^2\mathd\mu_\bx\mathd \mu_{\B{z}}\right)\mathd\tau_\by\nonumber\\
&\le & C  \left(\int_{\M}\int_\M \left(\int_{\p\M}R_t(\by,\mathbf{z})\bar{R}_t(\bx,\by)\mathd\tau_\by\right)\left(u(\bx)-u(\B{z})\right)^2\mathd\mu_\bx\mathd \mu_{\B{z}}\right)\nonumber\\
&=&  C  \left(\int_{\M}\int_\M Q(\bx,\B{z})\left(u(\bx)-u(\B{z})\right)^2\mathd\mu_\bx\mathd \mu_{\B{z}}\right),\nonumber
  \end{eqnarray}
where
\begin{eqnarray}
Q(\bx,\B{z})=\int_{\p\M}R_t(\by,\mathbf{z})\bar{R}_t(\bx,\by)\mathd\tau_\by.\nonumber
\end{eqnarray}
Notice that $Q(\bx,\B{z})=0$ if $\|\by-\B{z}\|^2\ge 16t$, 
%since $R_t(\bx,\by)=\bar{R}_t(\bx,\by)=0$ as $\|\bx-\by\|^2\ge 4t$ 
and $|Q(\bx,\B{z})|\le C C_t/\sqrt{t}$. We have 
\begin{eqnarray}
  |Q(\bx,\B{z})|\le \frac{CC_t}{\sqrt{t}} R\left(\frac{\|\bx-\B{z}\|^2}{32t}\right).\nonumber
\end{eqnarray}
Then, we obtain the following estimate,
\begin{eqnarray}
\label{eq:est_Q}
 && \left| \left(\int_{\M}\int_\M Q(\bx,\B{z})\left(u(\bx)-u(\B{z})\right)^2\mathd\mu_\bx\mathd \mu_{\B{z}}\right)\right|\\
&\le & \left| \frac{C}{\sqrt{t}}\left(\int_{\M}\int_\M C_t R\left(\frac{\|\bx-\B{z}\|^2}{32t}\right)\left(u(\bx)-u(\B{z})\right)^2\mathd\mu_\bx\mathd \mu_{\B{z}}\right)\right|\nonumber\\
&\le & \left| \frac{C}{\sqrt{t}}\left(\int_{\M}\int_\M C_t R\left(\frac{\|\bx-\B{z}\|^2}{4t}\right)\left(u(\bx)-u(\B{z})\right)^2\mathd\mu_\bx\mathd \mu_{\B{z}}\right)\right|\nonumber\\
&\le&C\sqrt{t}\left(\left|\int_\M u(\bx)r(\bx)\mathd\mu_\bx\right| + \frac{1}{\beta}\left|\int_\M u(\bx)\left(\int_{\p\M}\bar{R}_t(\bx,\by)u(\by)\mathd\tau_\by\right)\mathd\mu_\bx\right|\right)\nonumber\\
&\le &C\sqrt{t}\|u\|_{L^2(\M)}\|r\|_{L^2(\M)} + \frac{C\sqrt{t}}{\beta}\left|\int_\M u(\bx)\left(\int_{\p\M}\bar{R}_t(\bx,\by)u(\by)\mathd\tau_\by\right)\mathd\mu_\bx\right|.\nonumber
\end{eqnarray}
 % \begin{eqnarray}
 % \left|\int_\M u(\bx)\left(\int_{\p\M}\left(\bar{R}_t(\bx,\by)-\frac{\bar{w}_t(\by)}{w_t(\by)}R_t(\bx,\by)\right)u(\by)\mathd\tau_\by\right)\mathd\mu_\bx\right|\nonumber\\
 % \le C\sqrt{t}\|u\|_{L^2(\p\M)}\|u\|_{H^1(\M)}
 % \end{eqnarray}
On the other hand,
  \begin{eqnarray}
&&    \int_\M u(\bx)\left(\int_{\p\M}\frac{\bar{w}_t(\by)}{w_t(\by)}R_t(\bx,\by)u(\by)\tau_\by\right)\mathd\mu_\bx\nonumber\\
&= & \int_{\p\M} \frac{\bar{w}_t(\by)}{w_t(\by)}u(\by)\left(\int_{\M}R_t(\bx,\by)(u(\bx)-u(\by))\mathd\mu_\bx\right)\mathd\tau_\by+\int_{\p\M}\bar{w}_t(\by)u^2(\by)\mathd\tau_\by \nonumber\\
&= & \int_{\p\M} \bar{w}_t(\by)u(\by)\left(v(\by)-u(\by)\right)\mathd\tau_\by+\int_{\p\M}\bar{w}_t(\by)u^2(\by)\mathd\tau_\by,
\nonumber
  \end{eqnarray}
where $v$ is the same as defined in \eqref{eqn:smooth_v}.
Since $u$ solves $K_t u=r(\bx)$, we have
\begin{eqnarray}
 w_t(\bx)u(\bx)=w_t(\bx)v(\bx)-\frac{2t}{\beta}\int_{\p\M}R_t(\bx,\by)u(\by)\mathd\tau_\by -t\,r(\bx).
\label{eqn:sol_u}
\end{eqnarray}
Then, we obtain
\begin{eqnarray}
&&  \int_{\p\M} \bar{w}_t(\by)u(\by)\left(v(\by)-u(\by)\right)\mathd\tau_\by\nonumber\\
&=&\int_{\p\M}\frac{\bar{w}_t(\by)}{w_t(\by)}u(\by)\left(\frac{2t}{\beta} \int_{\p\M}R_t(\bx,\by)u(\bx)\mathd\tau_\bx -t\,r(\by)\right)\mathd \tau_\by\nonumber\\
&\le & \frac{C\sqrt{t}}{\beta}\|u\|_{L^2(\p\M)}^2+Ct\|u\|_{L^2(\p\M)}\|r\|_{L^2(\p\M)}\nonumber\\
&\le & \frac{C\sqrt{t}}{\beta}\|u\|_{L^2(\p\M)}^2+Ct\|u\|_{L^2(\p\M)}\|r\|_{H^1(\M)}.\nonumber
\end{eqnarray}
Combining above estimates together, we have
\begin{eqnarray}
&&   \int_\M u(\bx)\left(\int_{\p\M}\bar{R}_t(\bx,\by)u(\by)\tau_\by\right)\mathd\mu_\bx\nonumber\\
&\ge & \int_{\p\M}\bar{w}_t(\by)u^2(\by)\mathd\tau_\by- \frac{C\sqrt{t}}{\beta}\|u\|_{L^2(\p\M)}^2-Ct\|u\|_{L^2(\p\M)}\|r\|_{H^1(\M)}\nonumber\\
&& -C\sqrt{t}\|u\|_{L^2(\M)}\|r\|_{L^2(\M)} -\frac{C\sqrt{t}}{\beta}\left|\int_\M u(\bx)\left(\int_{\p\M}\bar{R}_t(\bx,\by)u(\by)\mathd\tau_\by\right)\mathd\mu_\bx\right|.\nonumber
\end{eqnarray}
We can choose $\frac{\sqrt{t}}{\beta}$ small enough such that $\frac{C\sqrt{t}}{\beta}\le \min\{\frac{1}{2}, \frac{w_{\min}}{6}\}$, which gives us 
\begin{eqnarray}
&&   \int_\M u(\bx)\left(\int_{\p\M}\bar{R}_t(\bx,\by)u(\by)\mathd \tau_\by\right)\mathd\mu_\bx\nonumber\\
&\ge & \frac{2}{3}\int_{\p\M}\bar{w}_t(\by)u^2(\by)\mathd\tau_\by- \frac{C\sqrt{t}}{\beta}\|u\|_{L^2(\p\M)}^2-Ct\|u\|_{L^2(\p\M)}\|r\|_{H^1(\M)}-C\sqrt{t}\|u\|_{L^2(\M)}\|r\|_{L^2(\M)} \nonumber\\
&\ge& \frac{w_{\min}}{2}\|u\|_{L^2(\p\M)}^2-Ct\|u\|_{L^2(\p\M)}\|r\|_{H^1(\M)}-C\sqrt{t}\|u\|_{L^2(\M)}\|r\|_{L^2(\M)} \nonumber\\
&\ge& \frac{w_{\min}}{4}\|u\|_{L^2(\p\M)}^2-Ct^2\|r\|_{H^1(\M)}^2-C\sqrt{t}\|u\|_{L^2(\M)}\|r\|_{L^2(\M)}\nonumber
\end{eqnarray}
Substituting the above estimate to the first inequality \eqref{ineq:l2_dv}, we obtain
\begin{eqnarray}
\label{eq:est_dv_robin}
 &&  \|\nabla v\|_{L^2(\M)}+\frac{w_{\min}}{4\beta}\|u\|_{L^2(\p\M)}^2\\
&\le& -C\int_\M u(\bx)r(\bx)\mathd\mu_\bx+ 
\frac{Ct^2}{\beta}\|r\|_{H^1(\M)}^2+\frac{C\sqrt{t}}{\beta}\|u\|_{L^2(\M)}\|r\|_{L^2(\M)}\nonumber\\
&\le & C\|u\|_{L^2(\M)}\|r\|_{L^2(\M)}+\frac{Ct^2}{\beta}\|r\|_{H^1(\M)}^2.\nonumber
\end{eqnarray}
Here we require that $\frac{\sqrt{t}}{\beta}$ is bounded by a constant independent on $\beta$ and $t$. 
Now, using the representation of $u$ given in \eqref{eqn:sol_u}, we obtain
\begin{eqnarray}
&&  \|\nabla u\|_{L^2(\M)}^2+\frac{w_{\min}}{8\beta}\|u\|_{L^2(\p\M)}^2\nonumber\\
&\le&   C\|\nabla v\|_{L^2(\M)}^2+\frac{Ct^2}{\beta^2}\left\|\nabla \left(\frac{1}{w_t(\bx)}\int_{\p\M}R_t(\bx,\by)u(\by)\mathd\tau_\by\right)\right\|_{L^2(\M)}^2 \nonumber\\
&& +Ct^2\left\|\nabla \left(\frac{r(\bx)}{w_t(\bx)}\right)\right\|_{L^2(\M)}^2+\frac{w_{\min}}{8\beta}\|u\|_{L^2(\p\M)}^2\nonumber\\
&\le & C\|\nabla v\|_{L^2(\M)}^2+\left(\frac{C\sqrt{t}}{\beta^2}+\frac{w_{\min}}{8\beta}\right)\|u\|_{L^2(\p\M)}^2+Ct\|r\|_{L^2(\M)}^2+Ct^2\|r\|_{H^1(\M)}^2\nonumber\\
&\le & C\|\nabla v\|_{L^2(\M)}^2+\frac{w_{\min}}{4\beta}\|u\|_{L^2(\p\M)}^2+Ct\|r\|_{L^2(\M)}^2+Ct^2\|r\|_{H^1(\M)}^2\nonumber\\
&\le& C\|u\|_{L^2(\M)}\|r\|_{L^2(\M)}+Ct\|r\|_{L^2(\M)}^2+\frac{Ct^2}{\beta}\|r\|_{H^1(\M)}^2.\nonumber
\end{eqnarray}
Here we require that $\frac{C\sqrt{t}}{\beta}\le \frac{w_{\min}}{8}$ in the third inequality.
Furthermore, we have
\begin{eqnarray}
&&  \|u\|_{L^2(\M)}^2\le C\left(\|\nabla u\|_{L^2(\M)}^2+\frac{w_{\min}}{8\beta}\|u\|_{L^2(\p\M)}^2\right)\nonumber\\
&\le& C\|u\|_{L^2(\M)}\|r\|_{L^2(\M)}+Ct\|r\|_{L^2(\M)}^2+\frac{Ct^2}{\beta}\|r\|_{H^1(\M)}^2\nonumber\\
&\le & \frac{1}{2}\|u\|_{L^2(\M)}^2+C\|r\|_{L^2(\M)}^2+\frac{Ct^2}{\beta}\|r\|_{H^1(\M)}^2, \nonumber
\end{eqnarray}
which implies that
\begin{eqnarray}
   \|u\|_{L^2(\M)}\le C\left(\|r\|_{L^2(\M)}+\frac{t}{\sqrt{\beta}}\|r\|_{H^1(\M)}\right).\nonumber
\end{eqnarray}

% \begin{eqnarray}
%   \|\nabla v\|_{L^2(\M)}^2&\le& C \left<u,K_tu\right>\le  C\|u\|_{L^2(\M)}\|r\|_{L^2(\M)}+C\left(t^2\|r\|_{H^1(\M)}^2+t\|u\|^2_{H^1(\M)}\right)\nonumber\\
% \end{eqnarray}
Finally, we obtain 
\begin{eqnarray}
\|\nabla u\|_{L^2(\M)}^2 &\le& C\|u\|_{L^2(\M)}\|r\|_{L^2(\M)}+\frac{Ct^2}{\beta}\|r\|_{H^1(\M)}^2 \nonumber \\
&\le& C\left(\|r\|_{L^2(\M)}+\frac{t}{\sqrt{\beta}}\|r\|_{H^1(\M)}\right)^2, \nonumber
\end{eqnarray}
which completes the proof.

\subsection{Stability of $K_t$ for boundary term (Theorem \ref{thm:regularity_boundary_robin})}

First, we denote
\begin{align*}
r(\bx)&=\int_{\p\M}\mathbf{b}(\by)\cdot(\bx-\by)\rhk  \mathd \tau_\by.
\end{align*}

The key point of the proof is to show that
  \begin{eqnarray}
\label{eq:est_boundary}
    %\left|\int_\M u(\bx)\left(\int_{\p\M}b^i(\by)\eta^j\bar{R}_t(\bx,\by)\mathd \tau_\by-\bar{b}\right)\mathd \mu_\bx\right|
	 %\le C\sqrt{t} \;\max_i\left(\|b^i\|_\infty\right) \|u\|_{H^1(\M)}.
   \left|\int_\M u(\bx)r(\bx) \mathd \mu_\bx\right|
	\le C\sqrt{t} \;\|\mathbf{b}\|_{H^1(\M)} \|u\|_{H^1(\M)}.
  \end{eqnarray}
%where $\|\mathbf{b}\|_{\infty}=\max_i \max_{\by\in \M}|b^i(\by)|$.

% First, notice that
% $$|\bar{r}|\le 
% C\sqrt{t}\;\|\mathbf{b}\|_{L^2(\p\M)}\le C\sqrt{t}\;\|\mathbf{b}\|_{H^1(\M)}.$$
% Then it is sufficient to show that
%   \begin{equation}
% \label{eq:est_boundary}
%     \left|\int_\M u(\bx)\left(\int_{\p\M}\mathbf{b}(\by)\cdot(\bx-\by)\bar{R}_t(\bx,\by)
%  \mathd \tau_\by\right) \mathd \mu_\bx\right|\le C\sqrt{t} \;\|\mathbf{b}\|_{H^1(\M)} \|u\|_{H^1(\M)}.
%   \end{equation}
Direct calculation gives that
\begin{eqnarray*}
  &&|2t\nabla\rrhk-(\bx-\by)\bar{R}_t(\bx,\by)|\le C|\bx-\by|^2\rhk,
\end{eqnarray*}
where $\rrhk=C_t\bar{\bar{R}}\left(\frac{\|\bx-\by\|^2}{4t}\right)$ and $\bar{\bar{R}}(r)=\int_{r}^{\infty}\bar{R}(s)\mathd s$.
This implies that
\begin{align}
\label{eq:est_boundary_1}
  &\left|\int_\M u(\bx) \int_{\p\M}\mathbf{b}(\by)\left((\bx-\by)\bar{R}_t(\bx,\by)+2t\nabla\rrhk\right)
 \mathd \tau_\by \mathd \mu_\bx\right|\\
\le & C\int_\M |u(\bx) |\int_{\p\M}|\mathbf{b}(\by)||\bx-\by|^2\rhk  \mathd \tau_\by\mathd \mu_\bx\nonumber\\
%\le &Ct\|\mathbf{b}\|_{}\int_\M |u(\bx)| \int_{\p\M}\rhk  \mathd \tau_\by\mathd \mu_\bx\nonumber\\
\le &Ct\|\mathbf{b}\|_{L^2(\p\M)} \left(\int_{\p\M}\left(\int_\M\rhk  \mathd\mu_\bx\right)
\left(\int_\M |u(\bx)|^2\rhk  \mathd \mu_\bx\right) \mathd \tau_\by\right)^{1/2}\nonumber\\
\le & Ct\|\mathbf{b}\|_{H^1(\M)} \left(\int_{\M} |u(\bx)|^2 
\left(\int_{\p\M} \rhk  \mathd \tau_\by\right)\mathd \mu_\bx\right)^{1/2}\nonumber\\
\le & Ct^{3/4}\|\mathbf{b}\|_{H^1(\M)}\|u\|_{L^2(\M)}.\nonumber
\end{align}
On the other hand, using the Gauss integral formula, we have
 \begin{eqnarray}
\label{eq:gauss_boundary}
&&    \int_\M u(\bx) \int_{\p\M}\mathbf{b}(\by)\cdot\nabla\rrhk  \mathd \tau_\by\mathd \mu_\bx\\
&=& \int_{\p\M} \int_{\M}u(\bx) T_\bx(\mathbf{b}(\by))\cdot\nabla\rrhk  \mathd \mu_\bx \mathd \tau_\by\nonumber\\
&=&\int_{\p\M} \int_{\p\M}\mathbf{n}(\bx)\cdot T_\bx(\mathbf{b}(\by))u(\bx)\rrhk   \mathd \tau_\bx\mathd \tau_\by\nonumber\\
&&-\int_{\p\M} \int_{\M}\text{div}_\bx[u(\bx) T_\bx(\mathbf{b}(\by))]\rrhk  \mathd \mu_\bx\mathd \tau_\by.\nonumber
  \end{eqnarray}
Here $T_\bx$ is the projection operator to the tangent space on $\bx$. To get the first equality, we use the 
fact that $\nabla\rrhk$ belongs to the tangent space on $\bx$, such that $\mathbf{b}(\by)\cdot\nabla\rrhk=T_\bx(\mathbf{b}(\by))\cdot\nabla\rrhk$
and $\bn(\bx)\cdot T_\bx(\mathbf{b}(\by))=\bn(\bx)\cdot \mathbf{b}(\by)$ where $\bn(\bx)$ is the out normal of $\p\M$ at $\bx\in \p\M$.

For the first term, we have
  \begin{align}
\label{eq:est_boundary_2}
&    \left|\int_{\p\M} \int_{\p\M}\mathbf{n}(\bx)\cdot T_\bx(\mathbf{b}(\by))u(\bx)\rrhk   \mathd \tau_\bx\mathd \tau_\by\right|\\
= &\left|\int_{\p\M} \int_{\p\M}\mathbf{n}(\bx)\cdot \mathbf{b}(\by)u(\bx)\rrhk   \mathd \tau_\bx\mathd \tau_\by\right|\nonumber\\
\le &C\|\mathbf{b}\|_{L^2(\p\M)}  \left(\int_{\p\M} \left(\int_{\p\M}|u(\bx)|\rrhk  \mathd \tau_\bx\right)^2
 \mathd \tau_\by\right)^{1/2}\nonumber\\
\le &C\|\mathbf{b}\|_{H^1(\M)}  \left(\int_{\p\M} \left(\int_{\p\M}\rrhk  \mathd \tau_\bx\right)
 \left(\int_{\p\M}|u(\bx)|^2\rrhk  \mathd \tau_\bx\right)
 \mathd \tau_\by\right)^{1/2}\nonumber\\
% &\le &Ct^{-1/4}\; \max_i\left(\|b^i\|_\infty\right) \left(\int_{\p\M} 
%  \left(\int_{\p\M}\rrhk\mathd \tau_\by\right)|u(\bx)|^2
% \mathd \tau_\bx\right)^{1/2}\nonumber\\
\le &Ct^{-1/2}\; \|\mathbf{b}\|_{H^1(\M)} \|u\|_{L^2(\p\M)}\le Ct^{-1/2}\; \|\mathbf{b}\|_{H^1(\M)} \|u\|_{H^1(\M)}.\nonumber
  \end{align}
We can also bound the second term on the right hand side of \eqref{eq:gauss_boundary}. By using the assumption that $\M\in C^\infty$, we have
\begin{align*}
&|  \text{div}_\bx[u(\bx) T_\bx(\mathbf{b}(\by))]|\nonumber\\
\le& |\nabla u(\bx)||T_\bx(\mathbf{b}(\by))|| |
+|u(\bx)||\text{div}_\bx[T_\bx(\mathbf{b}(\by))]|| |+|\nabla  ||u(\bx)T_\bx(\mathbf{b}(\by))| \\
\le & C(|\nabla u(\bx)|+|u(\bx)|)|\mathbf{b}(\by)|% \\
% \le& C \|\mathbf{b}\|_{\infty} (|\nabla u(\bx)|+|u(\bx)|)
\end{align*}
where the constant $C$ depends on the curvature of the manifold $\M$.

Then, we have
  \begin{eqnarray}
\label{eq:est_boundary_3}
&&    \left|\int_{\p\M} \int_{\M}\text{div}_\bx[u(\bx)T_\bx(\mathbf{b}(\by))]\rrhk   \mathd \mu_\bx\mathd \tau_\by\right|\\
&\le &C   \int_{\p\M}\mathbf{b}(\by)  \int_{\M}(|\nabla u(\bx)|+|u(\bx)|)\rrhk  \mathd \mu_\bx\mathd \tau_\by\nonumber\\
% &\le &C \max_i\left(\|b^i\|_\infty\right) \left(\int_{\p\M} \left(\int_\M \rrhk \mathd\mu_\bx\right)
% \left(\int_{\M}(|\nabla u(\bx)|^2+|u(\bx)|^2)\rrhk\mathd \mu_\bx\right)\mathd \tau_\by\right)^{1/2}\nonumber\\
&\le &C \|\mathbf{b}\|_{L^2(\p\M)}  \left(\int_{\M} (|\nabla u(\bx)|^2+|u(\bx)|^2) 
\left(\int_{\p\M}\rrhk  \mathd \tau_\by\right)\mathd \mu_\bx\right)^{1/2}\nonumber\\
\quad\quad&\le &C t^{-1/4}\;\|\mathbf{b}\|_{H^1(\M)}  \|u\|_{H^1(\M)}.\nonumber
  \end{eqnarray}
Then, the inequality \eqref{eq:est_boundary} is obtained from \eqref{eq:est_boundary_1},
\eqref{eq:gauss_boundary}, \eqref{eq:est_boundary_2} and \eqref{eq:est_boundary_3}.

Following the proof of Theorem \ref{thm:regularity_robin}, in \eqref{eq:est_Q} and \eqref{eq:est_dv_robin}, we bound 
$\left|\int_{\M}u(\bx)r(\bx)\mathd\bx\right|$ by $C\sqrt{t}\;\|\mathbf{b}\|_{H^1(\M)}\|u\|_{H^1(\M)}$, which implies that
\begin{eqnarray}
&&  \|\nabla u\|_{L^2(\M)}^2+\frac{w_{\min}}{8\beta}\|u\|_{L^2(\p\M)}^2\nonumber\\
&\le& C\sqrt{t}\;\|\mathbf{b}\|_{H^1(\M)}\|u\|_{H^1(\M)} +Ct\|r\|_{L^2(\M)}^2+\frac{Ct^2}{\beta}\|r\|_{H^1(\M)}^2\nonumber\\
&\le &C\|\mathbf{b}\|_{H^1(\M)}\left(\sqrt{t}\|u\|_{H^1(\M)} +t\right)\nonumber
\end{eqnarray}
where we use the estimates that
\begin{eqnarray*}
  \|r(\bx)\|_{L^2(\M)}&\le& Ct^{1/4}\|\mathbf{b}\|_{H^1(\M)},\\
  \|r(\bx)\|_{H^1(\M)}&\le& Ct^{-1/4}\|\mathbf{b}\|_{H^1(\M)}.
\end{eqnarray*}
Then, using the fact that 
\begin{eqnarray}
 \|u\|_{L^2(\M)}^2\le C\left(\|\nabla u\|_{L^2(\M)}^2+\frac{w_{\min}}{8\beta}\|u\|_{L^2(\p\M)}^2\right),\nonumber
\end{eqnarray}
we have
  \begin{eqnarray}
\|u\|_{H^1(\M)}^2\le C\|\mathbf{b}\|_{H^1(\M)}\left(\sqrt{t}\|u\|_{H^1(\M)} +t\right),
\nonumber
\end{eqnarray}
which completes the proof.
\section{Error analysis of the discretization (Theorem \ref{thm:dis_error_robin})}
In this section, we estimate the discretization error introduced by approximating the integrals in
\eqref{eq:integral-homo}, that is to prove Theorem \ref{thm:dis_error_robin}.
To simplify the notation, we introduce two intermediate operators defined as follows,
\begin{align}
L_{t, h}u(\bx) =& \invt\sum_{\bfp_j \in P} \hkxpj(u(\bx) - u(\bfp_j))V_j,
\label{eqn:laplace_dis} 
\\
K_{t, h}u(\bx) =& \invt\sum_{\bfp_j \in P} \hkxpj(u(\bx) - u(\bfp_j))V_j+\frac{2}{\beta}\sum_{\bs_j\in S}\rhkxsj u(\bfs_j)A_j. 
\label{eqn:laplace_dis_dirichlet}
\end{align}
If  $u_{t,h}=I_{\mathbf{f}}(\mathbf{u})$ with $\mathbf{u}$ satisfying Equation \eqref{eqn:dis-homo}. 
One can verify that the following two equations are satisfied,
\begin{eqnarray}
K_{t, h}u_{t,h}(\bx) &=&  \sum_{\bfp_j \in P} \rhkxpj f(\bfp_j)V_j.
\label{eqn:integral_dis_interp_dirichlet}
\end{eqnarray}

The following lemma is needed for proving Theorem \ref{thm:dis_error_robin}. Its proof is deferred to appendix.
\begin{lemma} Suppose $\mathbf{u}=(u_1,\cdots,u_n)^t$ satisfies equation \eqref{eqn:dis-homo}, there exist constants $C, T_0, r_0$ 
only depend on $\M$ and $\p\M$, such that
\begin{eqnarray}
  \left(\sum_{i=1}^{n}u_i^2V_i\right)^{1/2}+t^{1/4}\left(\sum_{l\in I_S}u_l^2A_l\right)^{1/2}
\le  C\|I_{\mathbf{f}}(\bfu)\|_{H^1(\M)}+C\sqrt{h}\,t^{3/4}\|f\|_{\infty},\nonumber
\end{eqnarray}
as long as $t\le T_0, \frac{\sqrt{t}}{\beta}\le r_0, \frac{h}{t^{3/2}}\le r_0$, $I_S=\{1\le l\le n: \bfp_l\in S\}$.
\label{lem:diff_dis}
\end{lemma}

\begin{proof} {\it of Theorem \ref{thm:dis_error_robin}}

Denote
\begin{align}
u_{t,h}(\bx)=I_{\mathbf{f}}(\bfu)=\frac{1}{w_{t,h}(\bx)}\left(\sum_{\bfp_j\in P}R_t(\bx,\bfp_j)u_jV_j-\frac{2t}{\beta}
\sum_{\bs_j\in S}\bar{R}_t(\bx,\bfs_j)u_jA_j
+t\sum_{\bfp_j\in P}\bar{R}_t(\bx,\bfp_j)f_jV_j\right),\nonumber
\end{align}
where $\bfu=(u_1,\cdots,u_N)^t$ solves Equation \eqref{eqn:dis_dirichlet}, $f_j = f(\bfp_j)$ and $w_{t,h}(\bx)=\sum_{\bfp_j\in P}R_t(\bx,\bfp_j)V_j$.  
For convenience, we set
\begin{align*}
  a_{t,h}(\bx)=&\frac{1}{w_{t,h}(\bx)}\sum_{\bfp_j\in P}R_t(\bx,\bfp_j)u_jV_j, \\
c_{t,h}(\bx)=&\frac{t}{w_{t,h}(\bx)}\sum_{\bfp_j\in P}\bar{R}_t(\bx,\bfp_j)f(\bfp_j)V_j,\\
d_{t,h}(\bx)=&-\frac{2t}{\beta w_{t,h}(\bx)}\sum_{\bfs_j\in S}\bar{R}_t(\bx,\bfs_j)u_jA_j.  
\end{align*}

Next we upper bound the approximation error $K_t(u_{t,h}) - K_{t,h}(u_{t, h})$.
Since $u_{t,h} = a_{t, h} +c_{t, h}+d_{t,h}$, we only need to upper bound the approximation error
for $a_{t, h}, c_{t, h}$ and $d_{t,h}$ separately.  
For $c_{t,h}$, 
\begin{eqnarray}
&&\left|\left(K_tc_{t, h} - K_{t,h}c_{t, h}\right)(\bx)\right|\nonumber \\
% &=& \frac{1}{t} \left|\int_{\mathcal{M}}R_t(\bx,\by)(c_{t,h}(\bx)-c_{t,h}(\by))  \mathd \mu_\by-\sum_{j}R_t(\bx,\bfp_j)(c_{t,h}(\bx)-c_{t,h}(\bfp_j))V_j\right|\nonumber\\
% &&+\frac{2}{\beta}\left|\int_{\p\M}\bar{R}_t(\bx,\by)b_{t,h}(\by)\mathd \tau_\by-\sum_{j\in S}\bar{R}_t(\bx,\bx_j)c_{t,h}(\bx_j)A_j\right|\nonumber\\
&\le &\frac{1}{t} \left|c_{t,h}(\bx)\right|\left|\int_{\mathcal{M}}R_t(\bx,\by) \mathd \mu_\by-\sum_{\bfp_j\in P}R_t(\bx,\bfp_j)V_j\right|\nonumber\\
&&+ \frac{1}{t}\left|\int_{\mathcal{M}}R_t(\bx,\by)c_{t,h}(\by)  \mathd \mu_\by-\sum_{\bfp_j\in P}R_t(\bx,\bfp_j)c_{t,h}(\bfp_j)V_j\right|\nonumber\\
&&+\frac{2}{\beta}\left|\int_{\p\M}\bar{R}_t(\bx,\by)c_{t,h}(\by)\mathd \tau_\by-\sum_{\bfs_j\in S}\bar{R}_t(\bx,\bfs_j)c_{t,h}(\bfs_j)A_j\right|\nonumber\\
&\le &\frac{Ch}{t^{3/2}}\left|c_{t,h}(\bx)\right|+ \frac{Ch}{t^{3/2}}\|c_{t,h}\|_{\infty}+\frac{Ch}{t}\|\nabla c_{t,h}\|_{\infty}+\frac{Ch}{\beta}\left(t^{-1}\|c_{t,h}\|_\infty
+t^{-1/2}\|\nabla c_{t,h}\|_\infty\right)\nonumber\\
%&\le &\frac{Ch}{t^{3/2}}t\|f\|_\infty+\frac{Ch}{t} t^{1/2}\|f\|_{\infty}+\frac{Ch}{\beta}\|f\|_\infty
%\nonumber\\
&\le &\frac{Ch}{\sqrt{t}}\left(1+\frac{\sqrt{t}}{\beta}\right)\|f\|_{\infty}.\nonumber
\end{eqnarray}
Now we upper bound $\|K_ta_{t, h} - K_{t,h}a_{t, h}\|_{L_2(\M)}$. First, we have 
\begin{eqnarray}
\label{eqn:a_1}
&&  \int_{\mathcal{M}}\left(a_{t,h}(\bx)\right)^2\left|\int_{\mathcal{M}}R_t(\bx,\by) \mathd \mu_\by-\sum_{\bfp_j\in P}R_t(\bx,\bfp_j)V_j\right|^2\mathd\mu_\bx\\
%&\le & \frac{Ch^2}{t}\int_{\mathcal{M}}\left(a_{t,h}(\bx)\right)^2\mathd\mu_\bx \nonumber\\
&\le & \frac{Ch^2}{t} \int_{\mathcal{M}} \left( \frac{1}{w_{t, h}(\bx)} \sum_{\bfp_j\in P}R_t(\bx,\bfp_j)u_jV_j \right)^2 \mathd\mu_\bx \nonumber \\
&\le & \frac{Ch^2}{t} \int_{\mathcal{M}} \left( \sum_{\bfp_j\in P}R_t(\bx,\bfp_j)u_j^2V_j  \right) \left( \sum_{\bfp_j\in P}R_t(\bx,\bfp_j)V_j  \right) \mathd\mu_\bx \nonumber \\
&\le & \frac{Ch^2}{t} \left( \sum_{\bfp_j\in P}u_j^2V_j \int_{\mathcal{M}}R_t(\bx,\bfp_j)  \mathd\mu_\bx    \right) 
\le  \frac{Ch^2}{t}\sum_{\bfp_j\in P}u_j^2V_j.\nonumber
\end{eqnarray}
Let 
\begin{eqnarray}
K_1 &=&  C_t\int_{\mathcal{M}}\frac{1}{w_{t, h}(\by)}R\left(\frac{|\bx-\by|^2}{4t}\right)R\left(\frac{|\bfp_i-\by|^2}{4t}\right) \mathd \mu_\by\nonumber\\
 &-&C_t\sum_{\bfp_j\in P} \frac{1}{w_{t, h}(\bfp_j)}R\left(\frac{|\bx-\bfp_j|^2}{4t}\right)R\left(\frac{|\bfp_i-\bfp_j|^2}{4t}\right)V_j. 
\nonumber
\end{eqnarray}
We have $|K_1|<\frac{Ch}{t^{1/2}}$ for some constant $C$ independent of $t$. In addition, notice that
only when $|\bx-\bfp_i|^2\leq 16t $ is $K_1\neq 0$, which implies 
\begin{eqnarray}
|K_1| \leq \frac{1}{\delta_0}|K_1|R\left(\frac{|\bx-\bfp_i|^2}{32t}\right). \nonumber
\end{eqnarray}
Then we have
\begin{eqnarray}
\label{eqn:a_2}
&&\int_{\mathcal{M}}\left|\int_{\mathcal{M}}R_t(\bx,\by)a_{t,h}(\by)  \mathd \mu_\by-\sum_{\bfp_j\in P}R_t(\bx,\bfp_j)a_{t,h}(\bfp_j)V_j\right|^2\mathd\mu_\bx\\
&=& \int_{\mathcal{M}}\left(\sum_{i=1}^nC_tu_iV_i K_1 \right)^2\mathd\mu_\bx\nonumber\\
&\le &\frac{Ch^2}{t} \int_{\mathcal{M}}\left(\sum_{i=1}^n C_t|u_i|V_i R\left(\frac{|\bx-\bfp_i|^2}{32t}\right)  \right)^2 \mathd\mu_\bx \nonumber\\
&\le &\frac{Ch^2}{t} \int_{\mathcal{M}} \left(\sum_{i=1}^n C_t R\left(\frac{|\bx-\bfp_i|^2}{32t}\right) u^2_iV_i\right)  
\left(\sum_{i=1}^n C_t R\left(\frac{|\bx-\bfp_i|^2}{32t} \right)V_i  \right) \mathd\mu_\bx \nonumber\\
&\le &\frac{Ch^2}{t} \sum_{i=1}^n \left(\int_{\mathcal{M}} C_t R\left(\frac{|\bx-\bfp_i|^2}{32t}\right)  \mathd\mu_\bx \left(u^2_iV_i\right)\right)  
\le \frac{Ch^2}{t} \left(\sum_{i=1}^nu_i^2V_i\right). 
\nonumber
\end{eqnarray} 

Let 
\begin{eqnarray}
K_2 &=&  C_t\int_{\p\mathcal{M}}\frac{1}{w_{t, h}(\by)}\bar{R}\left(\frac{|\bx-\by|^2}{4t}\right)R\left(\frac{|\bfp_i-\by|^2}{4t}\right) \mathd \tau_\by\nonumber\\
 &-&C_t\sum_{\bfs_j\in S} \frac{1}{w_{t, h}(\bfs_j)}\bar{R}\left(\frac{|\bx-\bfs_j|^2}{4t}\right)R\left(\frac{|\bfp_i-\bfs_j|^2}{4t}\right)A_j.\nonumber 
\end{eqnarray}
We have $|K_2|<\frac{Ch}{t}$ for some constant $C$ independent of $t$. In addition, notice that
only when $|\bx-\bfp_i|^2\leq 16t $ is $K_2\neq 0$, which implies 
\begin{eqnarray}
|K_2| \leq \frac{1}{\delta_0}|K_2|R\left(\frac{|\bx-\bfp_i|^2}{32t}\right). \nonumber
\end{eqnarray}
Then
\begin{eqnarray}
\label{eqn:a_2_boundary}
&&\int_{\mathcal{M}}\left|\int_{\p\mathcal{M}}\bar{R}_t(\bx,\by)a_{t,h}(\by)  \mathd \tau_\by-\sum_{\bfs_j\in S}\bar{R}_t(\bx,\bfs_j)a_{t,h}(\bfs_j)A_j\right|^2\mathd\mu_\bx\\
&=& \int_{\mathcal{M}}\left(\sum_{i=1}^nC_tu_iV_i K_2 \right)^2\mathd\mu_\bx\nonumber\\
&\le &\frac{Ch^2}{t^2} \int_{\mathcal{M}}\left(\sum_{i=1}^n C_t|u_i|V_i R\left(\frac{|\bx-\bfp_i|^2}{32t}\right)  \right)^2 \mathd\mu_\bx \nonumber\\
&\le &\frac{Ch^2}{t^2} \int_{\mathcal{M}} \left(\sum_{i=1}^n C_t R\left(\frac{|\bx-\bfp_i|^2}{32t}\right) u^2_iV_i\right)  
\left(\sum_{i=1}^n C_t R\left(\frac{|\bx-\bfp_i|^2}{32t} \right)V_i  \right) \mathd\mu_\bx \nonumber\\
&\le &\frac{Ch^2}{t^2} \sum_{i=1}^n \left(\int_{\mathcal{M}} C_t R\left(\frac{|\bx-\bfp_i|^2}{32t}\right)  \mathd\mu_\bx \left(u^2_iV_i\right)\right)  \le
 \frac{Ch^2}{t^2} \left(\sum_{i=1}^nu_i^2V_i\right). \nonumber
\end{eqnarray} 

Combining Equation~\eqref{eqn:a_1},~\eqref{eqn:a_2} and ~\eqref{eqn:a_2_boundary}, 
\begin{eqnarray}
&&\|K_ta_{t, h} - K_{t,h}a_{t, h}\|_{L^2(\M)}%=\left(\int_M \left|\left(K_t(a_{t, h}) - K_{t,h}(a_{t, h})\right)(\bx)\right|^2 \mathd\mu_\bx\right)^{1/2} \nonumber\\
%&=&  \left(\int_{\mathcal{M}}\left|\int_{\mathcal{M}}R_t(\bx,\by)(u_{t,h}^1(\bx)-u_{t,h}^1(\by))  \mathd \mu_\by-\sum_{j}R_t(\bx,\bfp_j)(u_{t,h}^1(\bx)-u_{t,h}^1(\bfp_j))V_j\right|^2\mathd\mu_\bx \right)^{1/2}\nonumber\\
% &\le & \frac{1}{t}\left(\int_{\mathcal{M}}\left(a_{t,h}(\bx)\right)^2\left|\int_{\mathcal{M}}R_t(\bx,\by) \mathd \mu_\by-\sum_{j}R_t(\bx,\bfp_j)V_j\right|^2\mathd\mu_\bx\right)^{1/2}\nonumber\\
% &&+  \frac{1}{t}\left(\int_{\mathcal{M}}\left|\int_{\mathcal{M}}R_t(\bx,\by)a_{t,h}(\by)  \mathd \mu_\by-\sum_{j}R_t(\bx,\bfp_j)a_{t,h}(\bfp_j)V_j\right|^2\mathd\mu_\bx
% \right)^{1/2} \nonumber \\
% &&+\frac{2}{\beta}\left(\int_{\mathcal{M}}\left|\int_{\p\mathcal{M}}\bar{R}_t(\bx,\by)a_{t,h}(\by)  \mathd \tau_\by-\sum_{j}\bar{R}_t(\bx,\bfp_j)a_{t,h}(\bfp_j)A_j\right|^2\mathd\mu_\bx\right)^{1/2}
% \nonumber\\
\le  \frac{Ch}{t^{3/2}}\left(1+\frac{\sqrt{t}}{\beta}\right) \left(\sum_{i=1}^nu_i^2V_i\right)^{1/2}\nonumber
\end{eqnarray}

Now we upper bound $\|K_td_{t, h} - K_{t,h}d_{t, h}\|_{L_2}$. We have 
\begin{eqnarray}
\label{eqn:d_1}
&&  \int_{\mathcal{M}}\left(d_{t,h}(\bx)\right)^2\left|\int_{\mathcal{M}}\bar{R}_t(\bx,\by) \mathd \tau_\by-\sum_{\bfp_j\in P}\bar{R}_t(\bx,\bfp_j)V_j\right|^2\mathd\mu_\bx\\
&\le & \frac{Ch^2}{t^2}\int_{\mathcal{M}}\left(d_{t,h}(\bx)\right)^2\mathd\mu_\bx \nonumber\\
&\le & \frac{Ch^2t}{\beta^2} \int_{\mathcal{M}} \left( \frac{1}{w_{t, h}(\bx)} \sum_{\bfs_j\in S}\bar{R}_t(\bx,\bfs_j)u_jA_j \right)^2 \mathd\mu_\bx \nonumber \\
&\le & \frac{Ch^2t}{\beta^2} \int_{\mathcal{M}} \left( \sum_{\bfs_j\in S}\bar{R}_t(\bx,\bfs_j)u_j^2A_j  \right) 
\left( \sum_{\bfs_j\in S}\bar{R}_t(\bx,\bfs_j)A_j  \right) \mathd\mu_\bx \nonumber \\
&\le & \frac{Ch^2\sqrt{t}}{\beta^2} \left( \sum_{j\in I_S}u_j^2A_j \int_{\mathcal{M}}\bar{R}_t(\bx,\bfp_j)  \mathd\mu_\bx    \right) 
\le  \frac{Ch^2\sqrt{t}}{\beta^2}\sum_{j\in I_S}u_j^2A_j.\nonumber
\end{eqnarray}
where $I_S=\{1\le l\le n: \bfp_l\in S\}$.

Let 
\begin{eqnarray}
K_3 &=&  C_t\int_{\mathcal{M}}\frac{1}{w_{t, h}(\by)}R\left(\frac{|\bx-\by|^2}{4t}\right)\bar{R}\left(\frac{|\bfp_i-\by|^2}{4t}\right) \mathd \mu_\by\nonumber\\
 &-&C_t\sum_{\bfp_j\in P} \frac{1}{w_{t, h}(\bfp_j)}R\left(\frac{|\bx-\bfp_j|^2}{4t}\right)\bar{R}\left(\frac{|\bfp_i-\bfp_j|^2}{4t}\right)V_j. \nonumber
\end{eqnarray}
We have $|K_3|<\frac{Ch}{t^{1/2}}$ for some constant $K_3$ independent of $t$. In addition, notice that
only when $|\bx-\bfp_i|^2\leq 16t $ is $K_3\neq 0$, which implies 
\begin{eqnarray}
|K_3| \leq \frac{1}{\delta_0}|C|R\left(\frac{|\bx-\bfp_i|^2}{4t}\right). \nonumber
\end{eqnarray}
Then we have
\begin{eqnarray}
\label{eqn:d_2}
&&\int_{\mathcal{M}}\left|\int_{\mathcal{M}}R_t(\bx,\by)d_{t,h}(\by)  \mathd \mu_\by-\sum_{\bfp_j\in P}R_t(\bx,\bfp_j)d_{t,h}(\bfp_j)V_j\right|^2\mathd\mu_\bx\\
&=&\frac{4t^2}{\beta^2} \int_{\mathcal{M}}\left(\sum_{i\in I_S}C_tu_iA_i K_3 \right)^2\mathd\mu_\bx\nonumber\\
&\le &\frac{Ch^2t}{\beta^2} \int_{\mathcal{M}}\left(\sum_{i\in I_S} C_t|u_i|A_i R\left(\frac{|\bx-\bfp_i|^2}{32t}\right)  \right)^2 \mathd\mu_\bx \nonumber\\
&\le &\frac{Ch^2t}{\beta^2} \int_{\mathcal{M}} \left(\sum_{i\in I_S} C_t R\left(\frac{|\bx-\bfp_i|^2}{32t}\right) u^2_iA_i\right)  
\left(\sum_{i\in I_S} C_t R\left(\frac{|\bx-\bfp_i|^2}{32t} \right)A_i  \right) \mathd\mu_\bx \nonumber\\
&\le &\frac{Ch^2\sqrt{t}}{\beta^2} \sum_{i\in I_S} \left(\int_{\mathcal{M}} C_t R\left(\frac{|\bx-\bfp_i|^2}{32t}\right)  \mathd\mu_\bx \left(u^2_iA_i\right)\right)  \le
 \frac{Ch^2\sqrt{t}}{\beta^2} \left(\sum_{i\in I_S}u_i^2A_i\right). \nonumber
\end{eqnarray} 

Let 
\begin{eqnarray}
K_4 &=&  C_t\int_{\p\mathcal{M}}\frac{1}{w_{t, h}(\by)}\bar{R}\left(\frac{|\bx-\by|^2}{4t}\right)\bar{R}\left(\frac{|\bfp_i-\by|^2}{4t}\right) \mathd \tau_\by\nonumber\\
 &-&C_t\sum_{\bfs_j\in S} \frac{1}{w_{t, h}(\bfs_j)}\bar{R}\left(\frac{|\bx-\bfs_j|^2}{4t}\right)\bar{R}\left(\frac{|\bfp_i-\bfs_j|^2}{4t}\right)A_j. \nonumber
\end{eqnarray}
We have $|K_4|<\frac{Ch}{t}$ for some constant $C$ independent of $t$. In addition, notice that
only when $|\bx-\bfp_i|^2\leq 16t $ is $K_4\neq 0$, which implies 
\begin{eqnarray}
|K_4| \leq \frac{1}{\delta_0}|K_4|R\left(\frac{|\bx-\bfp_i|^2}{32t}\right). \nonumber
\end{eqnarray}
and
\begin{eqnarray}
\label{eqn:d_2_boundary}
&&\int_{\mathcal{M}}\left|\int_{\p\mathcal{M}}\bar{R}_t(\bx,\by)d_{t,h}(\by)  \mathd \tau_\by-\sum_{j}\bar{R}_t(\bx,\bfp_j)d_{t,h}(\bfp_j)A_j\right|^2\mathd\mu_\bx\\
&=& \frac{4t^2}{\beta^2}\int_{\mathcal{M}}\left(\sum_{i\in I_S}C_tu_iA_i K_4 \right)^2\mathd\mu_\bx\nonumber\\
&\le &\frac{Ch^2}{\beta^2} \int_{\mathcal{M}}\left(\sum_{i\in I_S} C_t|u_i|A_i R\left(\frac{|\bx-\bfp_i|^2}{32t}\right)  \right)^2 \mathd\mu_\bx \nonumber\\
&\le &\frac{Ch^2}{\beta^2} \int_{\mathcal{M}} \left(\sum_{i\in I_S} C_t R\left(\frac{|\bx-\bfp_i|^2}{32t}\right) u^2_iA_i\right)  
\left(\sum_{i\in I_S} C_t R\left(\frac{|\bx-\bfp_i|^2}{32t} \right)A_i  \right) \mathd\mu_\bx \nonumber\\
&\le &\frac{Ch^2}{\beta^{2}\sqrt{t}} \sum_{i\in I_S} \left(\int_{\mathcal{M}} C_t R\left(\frac{|\bx-\bfp_i|^2}{32t}\right)  \mathd\mu_\bx \left(u^2_iA_i\right)\right)  \le
 \frac{Ch^2}{\beta^{2}\sqrt{t}} \left(\sum_{i\in I_S}u_i^2A_i\right). 
\nonumber
\end{eqnarray} 

Combining Equation~\eqref{eqn:d_1},~\eqref{eqn:d_2} and ~\eqref{eqn:d_2_boundary}, 
\begin{eqnarray}
&&\|K_td_{t, h} - K_{t,h}d_{t, h}\|_{L^2(\M)} % \nonumber \\
% &=&\left(\int_M \left|\left(K_t(d_{t, h}) - K_{t,h}(d_{t, h})\right)(\bx)\right|^2 \mathd\mu_\bx\right)^{1/2} \nonumber\\
% %&=&  \left(\int_{\mathcal{M}}\left|\int_{\mathcal{M}}R_t(\bx,\by)(u_{t,h}^1(\bx)-u_{t,h}^1(\by))  \mathd \mu_\by-\sum_{j}R_t(\bx,\bfp_j)(u_{t,h}^1(\bx)-u_{t,h}^1(\bfp_j))V_j\right|^2\mathd\mu_\bx \right)^{1/2}\nonumber\\
% &\le & \frac{1}{t}\left(\int_{\mathcal{M}}\left(d_{t,h}(\bx)\right)^2\left|\int_{\mathcal{M}}R_t(\bx,\by) \mathd \mu_\by-\sum_{j}R_t(\bx,\bfp_j)V_j\right|^2\mathd\mu_\bx\right)^{1/2}\nonumber\\
% &&+  \frac{1}{t}\left(\int_{\mathcal{M}}\left|\int_{\mathcal{M}}R_t(\bx,\by)d_{t,h}(\by)  \mathd \mu_\by-\sum_{j}R_t(\bx,\bfp_j)d_{t,h}(\bfp_j)V_j\right|^2\mathd\mu_\bx
% \right)^{1/2} \nonumber \\
% &&+\frac{2}{\beta}\left(\int_{\mathcal{M}}\left|\int_{\p\mathcal{M}}\bar{R}_t(\bx,\by)d_{t,h}(\by)  \mathd \tau_\by-\sum_{j}\bar{R}_t(\bx,\bfp_j)d_{t,h}(\bfp_j)A_j\right|^2\mathd\mu_\bx\right)^{1/2}
% \nonumber\\
% &\le &
\le
       \frac{Ch}{\beta t^{3/4}}\left(1+\frac{\sqrt{t}}{\beta}\right) \left(\sum_{i\in I_S}u_i^2A_i\right)^{1/2}\nonumber
\end{eqnarray}

Now assembling the parts together, we have the following upper bound.
\begin{eqnarray}
\label{eqn:elliptic_L_1}
&&\|K_tu_{t, h} - K_{t,h}u_{t, h}\|_{L^2(\M)} \\
%&\le& \|K_ta_{t, h} - K_{t,h}a_{t, h}\|_{L^2(\M)} + \|K_tb_{t, h} - K_{t,h}b_{t, h}\|_{L^2(\M)} + \|K_tc_{t, h} - K_{t,h}c_{t, h}\|_{L^2(\M)} \nonumber \\
%&\le& \frac{Ch}{t^{3/2}}\left(\frac{t^{1/2}}{\beta}\|g\|_\infty+t\|f\|_\infty+\left(\sum_i u_i^2V_i\right)^{1/2}+\frac{t^{3/4}}{\beta}\left(\sum_l u_l^2A_l\right)^{1/2}\right) \nonumber \\
&\le& \frac{Ch}{t^{3/2}}\left(\|g\|_{\infty} + t\|f\|_\infty+\left(\sum_{i=1}^n u_i^2V_i\right)^{1/2}+t^{1/4}\left(\sum_{l\in I_S} u_l^2A_l\right)^{1/2}\right).
\nonumber
\end{eqnarray}
At the same time, since $u_t$ and $u_{t,h}$ solve ~\eqref{eqn:integral_dirichlet} 
and \eqref{eqn:dis_dirichlet} respectively, 
we have 
\begin{eqnarray}
\label{eqn:elliptic_L_2}
&&\|K_t(u_{t}) - K_{t,h}(u_{t, h})\|_{L^2(\M)}  \\ 
&=&\left(\int_\M \left(\left(K_tu_{t} - K_{t,h}u_{t, h}\right)(\bx)\right)^2     \mathd\mu_\bx\right)^{1/2}  \nonumber \\
&\leq& % \frac{2}{\beta}\left( \int_\M \left(\int_{\p\mathcal{M}}\bar{R}_t(\bx,\by)g(\by) \mathd \tau_\by- \sum_{j}\bar{R}_t(\bx,\bfs_j)g(\bfs_j)A_j\right)^2 \mathd \mu_\bx\right)^{1/2}\nonumber\\
% &+& 
    \left( \int _\M \left( \int_{\mathcal{M}}\bar{R}_t(\bx,\by)f(\by) - \sum_{\bfp_j\in P}\bar{R}_t(\bx,\bfp_j)f(\bfp_j)V_j\right)^2 \mathd\mu_\bx\right)^{1/2} \nonumber \\
&\le& \frac{Ch}{t^{1/2}}\|f\|_\infty. \nonumber
\end{eqnarray}
From Equation~\eqref{eqn:elliptic_L_1} and~\eqref{eqn:elliptic_L_2}, we get
 \begin{equation}
\label{eqn:elliptic_K}
    \|K_tu_{t} - L_{t}u_{t, h}\|_{L^2(\M)}\le  \frac{Ch}{t^{3/2}}\left(\left(\sum_{i=1}^n u_i^2V_i\right)^{1/2}+t^{1/4}\left(\sum_{l\in I_S} u_l^2A_l\right)^{1/2}+t\|f\|_{\infty}\right).
  \end{equation}

Using the similar techniques, we can get the upper bound of $\|\nabla (K_tu_{t} - L_{t}u_{t, h})\|_{L_2(\M)}$ as following. 
\begin{equation}
\|\nabla \left(K_tu_{t} - L_{t}u_{t, h}\right)\|_{L^2(\M)}
\le \frac{Ch}{t^2}\left(t\|f\|_{C^1(\M)}+\left(\sum_{i=1}^{n} u_i^2V_i\right)^{1/2}+t^{1/4}\left(\sum_{l\in I_S} u_l^2A_l\right)^{1/2}\right).
\label{eqn:elliptic_dK}
\end{equation}

In the remaining of the proof, we only need to get a prior estimate of $\left(\sum_{i=1}^n u_i^2V_i\right)^{1/2}+t^{1/4}\left(\sum_{l\in I_S} u_l^2A_l\right)^{1/2}$.
First, using the estimate \eqref{eqn:elliptic_K} and \eqref{eqn:elliptic_dK} and the Theorem \ref{thm:regularity_robin}, we have
 \begin{eqnarray}
    \|u_{t,h}\|_{H^1(\M)}&\le&  \frac{Ch}{t^{3/2}}\left(\left(\sum_{i=1}^n u_i^2V_i\right)^{1/2}+t^{1/4}\left(\sum_{l\in I_S} u_l^2A_l\right)^{1/2}+t\|f\|_\infty\right)\nonumber\\
&&+C\|K_t u_t\|_{L^2(\M)}+Ct^{3/4}\|K_t u_t\|_{H^1(\M)}.
\label{eqn:uh_h1_comp}
  \end{eqnarray}
Using the relation that $K_t u_t=-\int_{\M} \bar{R}_t(\bx,\by)f(\by)\mu_\by$, it is easy to get that
\begin{eqnarray}
  \|K_t u_t\|_{L^2(\M)}&\le&C\|f\|_\infty,\\
  \|\nabla (K_t u_t)\|_{L^2(\M)}&\le&\frac{C}{t^{1/2}}\|f\|_\infty.
\end{eqnarray}
Substituting above estimates in \eqref{eqn:uh_h1_comp}, we have
  \begin{eqnarray}
    \|u_{t,h}\|_{H^1(\M)}\le  \frac{Ch}{t^{3/2}}\left(\left(\sum_{i=1}^n u_i^2V_i\right)^{1/2}+t^{1/4}\left(\sum_{l\in I_S} u_l^2A_l\right)^{1/2}+t\|f\|_\infty\right)
+C\|f\|_\infty.\nonumber
  \end{eqnarray}

Using Lemma \ref{lem:diff_dis}, we have
\begin{eqnarray}
&&  \left(\sum_{i=1}^nu_i^2V_i\right)^{1/2}+t^{1/4}\left(\sum_{l\in I_S}u_l^2A_l\right)^{1/2}\nonumber\\
&\le&  C\|u_{t,h}\|_{H^1(\M)}+C\sqrt{h}\left(t^{3/4}\|f\|_{\infty}+\|g\|_{\infty}\right)\nonumber\\
&\le&  \frac{Ch}{t^{3/2}}\left( t\|f\|_\infty+\left(\sum_{i=1}^n u_i^2V_i\right)^{1/2}+t^{1/4}\left(\sum_{l\in I_S} u_l^2A_l\right)^{1/2}\right)\nonumber\\
&&+C\|f\|_\infty+C\sqrt{h}\,t^{3/4}\|f\|_{\infty}
\end{eqnarray}
Using the assumption that $\frac{h}{t^{3/2}}$ is small enough such that $\frac{Ch}{t^{3/2}}\le \frac{1}{2}$, we have
\begin{eqnarray}
\label{eqn:uh_l2_comp}
  \left(\sum_{i=1}^nu_i^2V_i\right)^{1/2}+t^{1/4}\left(\sum_{l\in I_S}u_l^2A_l\right)^{1/2}
\le  C\|f\|_\infty
\end{eqnarray}
Then the proof is complete by substituting above estimate \eqref{eqn:uh_l2_comp} in \eqref{eqn:elliptic_K} and \eqref{eqn:elliptic_dK}.
\end{proof}

%%% Local Variables:
%%% mode: latex
%%% TeX-master:"paper_convergence_dirichlet"
%%% End:

%\section{Bounds on Solution Operators (Lemma \ref{lem:bounds_T})}
%\input{h1bound}

% \section{Numerical Results}
% \label{sec:experiment}
% \input{experiments}

\section{Discussion and Future Work}
\label{sec:discussion}

We have proved the convergence of the point integral method for the Poisson equation on manifolds with 
the Dirichlet boundary. 
In point integral method, the Dirichlet boundary can not be enforced directly. 
In this paper, we use Robin boundary to approximate the Dirichlet boundary and use point integral method to 
solve the Poisson equation with Robin boundary condition.

Another way to deal with the Dirichlet boundary condition in point integral method is using the volume 
constraint proposed by Du et.al. \cite{Du-SIAM}. The volume constraint has been integrated into the point integral method to enforce the Dirichlet boundary condition and the convergence has been proved \cite{shi-vc}.

\vspace{0.3in}
\noindent
{\bf Acknowledgments.}
This research was supported by NSFC Grant 11371220. 

\vspace{0.3in}
% \noindent
% {\bf Acknowledgments.}
% This research was supported by NSFC Grant 11371220. 
%and  National Basic Research Program of China (973 Program 2012CB825500 to J.S.).

\appendix
\label{sec:appendix}

\section{Proof of Lemma \ref{lem:diff_dis}}
\begin{proof} 
First, denote
\begin{eqnarray}
u_{t,h}(\bx)=I_{\bff}(\bfu)
=\frac{1}{w_{t,h}(\bx)}\left(\sum_{j=1}^nR_t(\bx,\bfp_j)u_jV_j-\frac{2t}{\beta}\sum_{\bs_j\in S}\bar{R}_t(\bx,\bfs_j)u_jA_j
+t\sum_{j=1}^n\bar{R}_t(\bx,\bfp_j)f_jV_j\right),\nonumber
\end{eqnarray}
where $f_j = f(\bfp_j)$ and $w_{t,h}(\bx)=\sum_{j=1}^nR_t(\bx,\bfp_j)V_j$ and $\bfu=(u_j)$ solves \eqref{eqn:dis_dirichlet}
 with $b=0$.
Let
\begin{eqnarray*}
  v_1(\bx)&=&\frac{1}{w_{t,h}(\bx)}\sum_{j=1}^nR_t(\bx,\bfp_j)u_jV_j,~\text{and}\\
v_2(\bx)&=&-\frac{2t}{\beta w_{t,h}(\bx)}\sum_{\bs_j\in I_S}\bar{R}_t(\bx,\bfs_j)u_jA_j,~\text{and}\\
v_3(\bx)&=&\frac{t}{w_{t,h}(\bx)}\sum_{j=1}^n\bar{R}_t(\bx,\bfp_j)f_jV_j, 
\end{eqnarray*}
and then $u_{t,h}=v_1+v_2+v_3$ and  
\begin{eqnarray}
  \left|\|u_{t,h}\|_{L^2(\M)}^2-\sum_{j=1}^nu_j^2V_j\right|&=&\left|\sum_{m,m'=1}^{3}\left(\int_\M v_m(\bx)v_{m'}(\bx)\mathd \mu_\bx
-\sum_{j=1}^nv_m(\bx_j)v_{m'}(\bx_j)V_j\right)\right|\nonumber\\
&\le &\sum_{m,m'=1}^{3}\left|\int_\M v_m(\bx)v_{m'}(\bx)\mathd \mu_\bx-\sum_{j=1}^nv_m(\bx_j)v_{m'}(\bx_j)V_j\right|.\nonumber
\end{eqnarray}
We now estimate these six terms in the above summation one by one. First, we consider the term with $m=m'=1$.
Denote
\begin{eqnarray}
A = \int_{\mathcal{M}}\frac{C_t}{w_{t,h}^2(\bx)}R\left(\frac{|\bx-\bfp_i|^2}{4t}\right)
R\left(\frac{|\bx-\bfp_l|^2}{4t}\right)\mathd\mu_\bx-\nonumber\\
\sum_{j=1}^n\frac{C_t}{w_{t,h}^2(\bfp_j)}R\left(\frac{|\bfp_j-\bfp_i|^2}{4t}\right)R\left(\frac{|\bfp_j-\bfp_l|^2}{4t}\right)V_j,\nonumber
\end{eqnarray}
and then $|A|\le \frac{Ch}{t^{1/2}}$. At the same time, notice that only when $|\bfp_i-\bfp_l|^2 <16t$ is $A\neq 0$. Thus we have 
\begin{eqnarray}
|A| \le \frac{1}{\delta_0} |A| R(\frac{|\bfp_i-\bfp_l|^2}{32t}), \nonumber
\end{eqnarray}
and
\begin{eqnarray}
&&\left|\int_{\mathcal{M}}v_1^2(\bx)\mathd\mu_\bx-\sum_{j=1}^nv_1^2(\bfp_j)V_j\right| \nonumber\\
&\leq&  \sum_{i,l=1}^n|C_{t}u_iu_lV_iV_l| |A| \nonumber\\
&\le &\frac{Ch}{t^{1/2}} \sum_{i,l=1}^n\left|C_{t} R\left(\frac{|\bfp_i-\bfp_l|^2}{32t}\right) u_iu_lV_iV_l \right|\nonumber\\
&\le &\frac{Ch}{t^{1/2}} \sum_{i=1}^n \left(\sum_{l=1}^nC_{t} R\left(\frac{|\bfp_i-\bfp_l|^2}{32t}\right)V_l\right)^{1/2}\left(\sum_{l=1}^nC_{t} R\left(\frac{|\bfp_i-\bfp_l|^2}{32t}\right)
u_l^2V_l\right)^{1/2} u_iV_i \nonumber\\
&\le &\frac{Ch}{t^{1/2}} \left(\sum_{i=1}^n \sum_{l=1}^nC_{t} R\left(\frac{|\bfp_i-\bfp_l|^2}{32t}\right)
u_l^2V_lV_i\right)^{1/2} \left(\sum_{i=1}^nu_i^2V_i\right)^{1/2}.\nonumber\\
&= &\frac{Ch}{t^{1/2}} \left(\sum_{l=1}^nu_l^2V_l \sum_{i=1}^nC_{t} R\left(\frac{|\bfp_i-\bfp_l|^2}{32t}\right)
V_i\right)^{1/2} \left(\sum_{i=1}^nu_i^2V_i\right)^{1/2}\nonumber\\
\quad\quad\quad&\le & \frac{Ch}{t^{1/2}} \sum_{i=1}^nu_i^2V_i.\nonumber
\end{eqnarray}
Using a similar argument, we can obtain the following estimates for the remaining terms,
\begin{eqnarray*}
\label{eqn:u_12}
\left|\int_{\mathcal{M}}v_1(\bx)v_2(\bx)\mathd\mu_\bx-\sum_{j=1}^nv_1(\bfp_j)v_2(\bfp_j)V_j\right|&\le & \frac{Cht^{1/4}}{\beta} \left(\sum_{i=1}^nu_i^2V_i\right)^{1/2}
\left(\sum_{l\in I_S}u_l^2A_l\right)^{1/2},~\text{and}\nonumber\\
\label{eqn:u_13}
\left|\int_{\mathcal{M}}v_1(\bx)v_3(\bx)\mathd\mu_\bx-\sum_{j=1}^nv_1(\bfp_j)v_3(\bfp_j)V_j\right|&\le & Cht^{1/2} \left(\sum_{i=1}^nu_i^2V_i\right)^{1/2}
\left(\sum_{j=1}^nf_j^2V_j\right)^{1/2},~\text{and}\nonumber\\
\label{eqn:u_22}
\left|\int_{\mathcal{M}}v_2^2(\bx)\mathd\mu_\bx-\sum_{j=1}^nv_2^2(\bfp_j)V_j\right|&\le & \frac{Cht}{\beta^2}
\sum_{l\in I_S}u_l^2A_l,~\text{and}\nonumber\\
\label{eqn:u_23}
\left|\int_{\mathcal{M}}v_2(\bx)v_3(\bx)\mathd\mu_\bx-\sum_{j=1}^nv_2(\bfp_j)v_3(\bfp_j)V_j\right|&\le & \frac{Cht^{5/4}}{\beta} \left(\sum_{l\in I_S}u_l^2A_l\right)^{1/2}
\left(\sum_{j=1}^nf_j^2V_j\right)^{1/2},~\text{and}\nonumber\\
\label{eqn:u_33}
\left|\int_{\mathcal{M}}v_3^2(\bx)\mathd\mu_\bx-\sum_{j=1}^nv_3^2(\bfp_j)V_j\right|&\le & Cht^{3/2} \sum_{j=1}^nf_j^2V_j\nonumber.\\
\end{eqnarray*}
Assembling all the above estimates together, we obtain
\begin{eqnarray}
  \left|\|u_{t,h}\|^2_{L^2(\M)}-\sum_{i=1}^nu_i^2V_i\right|\le \frac{Ch}{t^{1/2}}\left(\sum_{i=1}^nu_i^2V_i+t^{1/2}\sum_{l\in I_S}u_l^2A_l+t^2\|f\|_{\infty}^2\right).\nonumber
\end{eqnarray}
Similarly, we have
\begin{eqnarray}
  \left|\|u_{t,h}\|^2_{L^2(\p\M)}-\sum_{l\in I_S}u_l^2A_l\right|\le \frac{Ch}{t}\left(\sum_{i=1}^nu_i^2V_i+t^{1/2}\sum_{l\in I_S}u_l^2A_l+t^2\|f\|_{\infty}^2
\right).\nonumber
\end{eqnarray}
Using the assumption that $\frac{h}{t^{1/2}}$ is small enough such that $\frac{Ch}{t^{1/2}}\le \frac{1}{2}$, we obtain 
\begin{eqnarray}
  \sum_{i=1}^nu_i^2V_i+t^{1/2}\sum_{l\in I_S}u_l^2A_l&\le& 2\left(\|u_{t,h}\|^2_{L^2(\M)}+t^{1/2}\|u_{t,h}\|^2_{L^2(\p\M)}\right)
+Ch\left(t^{3/2}\|f\|_{\infty}^2\right)\nonumber\\
&\le & C\|u_{t,h}\|^2_{H^1(\M)}+Cht^{3/2}\|f\|_{\infty}^2,\nonumber
\end{eqnarray}
which implies that
\begin{eqnarray}
  \left(\sum_{i=1}^nu_i^2V_i\right)^{1/2}+t^{1/4}\left(\sum_{l\in I_S}u_l^2A_l\right)^{1/2}
\le  C\|u_{t,h}\|_{H^1(\M)}+C\sqrt{h}\,t^{3/4}\|f\|_{\infty}.\nonumber
\end{eqnarray}
\end{proof}
%%% Local Variables: 
%%% mode: latex
%%% TeX-master:"paper_convergence_dirichlet"
%%% End: 

\bibliographystyle{abbrv}
\bibliography{reference}

\end{document}